\documentclass[a4paper, 10pt]{amsart}

\usepackage[square, numbers, comma]{natbib} % Use the natbib reference package - read up on this to edit the reference style; if you want text (e.g. Smith et al., 2012) for the in-text references (instead of numbers), remove 'numbers'. Add sort&compress to sort references the way they appear in the bibliography, and compress references which are close in numbers.

\usepackage[usenames,dvipsnames,svgnames,table]{xcolor}
\usepackage{amsmath,amsthm,amssymb,amssymb,esint,verbatim,tabularx,graphicx}
\usepackage{bbm}
\usepackage{fancyhdr}
\usepackage{enumerate}
\usepackage{amsmath}
\usepackage{fancyhdr}
\usepackage{epic}
\usepackage{pgf,tikz}
\usetikzlibrary{arrows}

\usepackage[utf8]{inputenc}
\usepackage{color}
\usepackage{hyperref}
\usepackage{verbatim}
\usepackage{pdfpages}
\usepackage{empheq}

\makeatletter
\def\@tocline#1#2#3#4#5#6#7{\relax
  \ifnum #1>\c@tocdepth % then omit
  \else
    \par \addpenalty\@secpenalty\addvspace{#2}%
    \begingroup \hyphenpenalty\@M
    \@ifempty{#4}{%
      \@tempdima\csname r@tocindent\number#1\endcsname\relax
    }{%
      \@tempdima#4\relax
    }%
    \parindent\z@ \leftskip#3\relax \advance\leftskip\@tempdima\relax
    \rightskip\@pnumwidth plus4em \parfillskip-\@pnumwidth
    #5\leavevmode\hskip-\@tempdima
      \ifcase #1
       \or\or \hskip 1.8em \or \hskip 2em \else \hskip 3em \fi%
      #6\nobreak\relax
    \dotfill\hbox to\@pnumwidth{\@tocpagenum{#7}}\par
    \nobreak
    \endgroup
  \fi}
\makeatother
\usepackage{hyperref}

\hypersetup{urlcolor=blue, colorlinks=true} % Colors hyperlinks in blue - change to black if annoying

\usepackage[inner=2.5cm,outer=2.5cm,bottom=4cm,top=4cm]{geometry}

\newtheorem{theorem}{Theorem}[section]

\newtheorem{lemma}[theorem]{Lemma}
\newtheorem{proposition}[theorem]{Proposition}
\newtheorem{corollary}[theorem]{Corollary}
\newtheorem{definition}{Definition}[section]

\theoremstyle{remark}
\newtheorem{remark}[theorem]{Remark}
\theoremstyle{definition}

\numberwithin{equation}{section}

\newcommand{\R}{\ensuremath{\mathbb{R}}}
\newcommand{\N}{\ensuremath{\mathbb{N}}}

\newcommand{\Z}{\ensuremath{\mathbb{Z}}}

\newcommand{\veps}{\varepsilon}
\newcommand{\Div}{\mbox{div}}

\newcommand{\plap}{\ensuremath{\Delta_p}}
\newcommand{\Mmvp}{\mathcal{M}_r^p}
\newcommand{\MVPrh}{\Delta_{p}^h}

\newcommand{\U}{\mathcal{U}}

\newcommand{\G}{\tilde{G}}

\newcommand{\dd}{\,\mathrm{d}}

\newcommand{\lap}{\Delta}

\newcommand{\Grid}{\mathcal{G}_h}

\renewcommand{\vec}{}

\begin{document}

\title[Finite differences for the $p$-Laplacian]{A finite difference method\\  for the variational $p$-Laplacian}

\author[F.~del~Teso]{F\'elix del Teso}

\address[F. del Teso]{Departamento de An\'alisis Matem\'atico y Matem\'atica Aplicada, 
Universidad Complutense  de Madrid, 28040 Madrid, Spain} 

\email[]{fdelteso\@@{}ucm.es}

\urladdr{https://sites.google.com/view/felixdelteso}

\keywords{$p$-Laplacian, finite difference, mean value property, nonhomogeneous Dirichlet problem, viscosity solutions, dynamic programming principle.}

\author[E.~Lindgren]{Erik Lindgren}

\address[E. Lindgren]{Department of Mathematics, Uppsala University, Box 480, 751 06 Uppsala, Sweden}
\email[]{erik.lindgren\@@{}math.uu.se}
\urladdr{https://sites.google.com/view/eriklindgren}

\subjclass[2010]{
65N06,  %	Finite difference methods for boundary value problems involving PDEs
 35J60, %Nonlinear elliptic equations
 35J70, %Degenerate elliptic equations
 35J75, %Singular elliptic equations
 35J92, %Quasilinear elliptic equations with p-Laplacian 
 35D40, %Viscosity solutions
 35B05. %Oscillation, zeros of solutions, mean value theorems, etc. in context of 
 }

\begin{abstract}\noindent   
 We propose a new monotone finite difference discretization for the variational $p$-Laplace operator, 
\[
\plap u=\Div(|\nabla u|^{p-2}\nabla u),
\]
and present a convergent numerical scheme for related Dirichlet problems. The resulting nonlinear system is solved using two different methods: one based on Newton-Raphson and one explicit method.
Finally, we exhibit some numerical simulations supporting our theoretical results. 
 
 To the best of our knowledge, this is the first monotone finite difference discretization of the variational $p$-Laplacian
 and also the first time that nonhomogeneous problems for this operator can be treated numerically with a finite difference scheme.
\end{abstract}

\maketitle
{\large
\tableofcontents 
}

\section{Introduction and main results} \label{sec:intro}

In the recent paper \cite{dTLi20}, we studied a new\footnote{This mean value formula was independently derived for $p\geq 2$ in \cite{BS18},
see Proposition 2.10 and Theorem 2.12 therein.} mean value formula (MVF) for the variational  $p$-Laplace operator, 
\begin{equation}\label{eq:pLap}
\lap_p u = \Div(|\nabla u|^{p-2}\nabla u).
\end{equation}
With the notation $J_p(t):=|t|^{p-2}t$ for all $p>1$, the MVF, valid for any $C^2(\R^d)$ function, reads
\begin{equation}\label{eq:MVF}
\frac{1}{D_{d,p} r^p}\fint_{B_r} J_p(u(x+y)-u(x)) \dd y =\lap_p u(x) + o_r(1) \quad \textup{as} \quad r\to0^+.
\end{equation}
Here $D_{d,p}:=\frac{d}{2(d+p)}\fint_{\partial B_1} |y_1|^{p} \dd\sigma(y)$ and $B_r$ denotes the ball of radius $r>0$ centered at 0. 

The aim of this paper is to propose a new monotone finite difference discretization of the $p$-Laplacian based on the asymptotic expansion \eqref{eq:MVF}. We also propose a convergent numerical scheme associated to the nonhomogeneous Dirichlet problem 
\begin{empheq}[left=\empheqlbrace]{align}
-\lap_p u(x) =f(x),&\quad  x\in \Omega,\label{eq:BVP}\\
u(x)=g(x),& \quad  x\in \partial \Omega, \label{eq:BC}
\end{empheq}
The scheme results in a nonlinear system. We propose two methods to solve this system: 1) Newton-Raphson and 2) an explicit method, based on the convergence to a steady state of an evolution problem. We comment the advantages of each one in Section \ref{sec:NonlinearSis}. 
 Finally, we exhibit some numerical tests of the accuracy and convergence of the scheme.

To the best of our knowledge, this is the first monotone finite difference discretization of the variational $p$-Laplacian available in the literature and therefore the first time that nonhomogeneous problems of the form \eqref{eq:BVP}-\eqref{eq:BC} can be treated numerically via finite difference schemes. The monotonicity property (see Lemma \ref{lem:propscheme}) is crucial for the convergence of finite difference schemes in the context of viscosity solutions (see \cite{BS91}). 
It is also worth mentioning that, in contrast to the finite difference schemes for the normalized (or game theoretical) $p$-Laplacian considered earlier (see Section \ref{sec:relres}), our scheme is well suited for Newton-Raphson solvers, which is an advantage when it comes to solving a nonlinear system effectively.

\subsection{Main results} In order to describe our main results we need to introduce some notation. Given a discretization parameter $h>0$, consider the uniform grid defined by $\Grid:= h \Z^d=\{y_\alpha:=h\alpha \, : \, \alpha \in \Z^d\}$. 
Let $r>0$ and consider the following discrete operator
\begin{equation}\label{eq:discPlap}
\MVPrh\phi(x):= \frac{h^d}{D_{d,p}\, \omega_d\, r^{p+d}} \sum_{y_\alpha\in B_r} J_p(\phi(x+y_\alpha)-\phi(x)),
\end{equation}
where $\omega_d$ denotes the measure of the unit ball in $\R^d$.  Throughout the paper, we will assume the following relation between $h$ and $r$:
\begin{equation}\tag{H}
\label{eq:ass}
h=\begin{cases}
o(r^\frac{p}{p-1}), & \quad \textup{if} \quad p \in (1,3)\setminus\{2\},\\
o(r), &  \quad \textup{if} \quad p=2,\\
o(r^{\frac{3}{2}}),&  \quad \textup{if} \quad p\in [3,\infty).\\
\end{cases}
\end{equation}
Our first result regards the consistency of the discretization \eqref{eq:discPlap}.
\begin{theorem}\label{thm:MVF}
Let $p\in(1,\infty)$, $x\in \R^d$  and $\phi\in C^2(B_R(x))$ for some $R>0$.  Assume \eqref{eq:ass}. Then
\begin{equation*}
\MVPrh \phi (x)=\lap_p \phi(x)+o_h(1) \quad \textup{as} \quad r\to0^+.
\end{equation*}
\end{theorem}
Our second result concerns the finite difference numerical scheme for \eqref{eq:BVP}-\eqref{eq:BC} induced by the discretization \eqref{eq:discPlap}. More precisely, let $\partial\Omega_r:=\{x\in \Omega^c \, :\, \textup{dist}(x,\Omega)\leq r \}$, $\Omega_r=\Omega\cup \partial\Omega_r$ and let $G$ be a continuous extension of $g$ from $\partial \Omega$ to $\partial \Omega_r$. Consider $u_h:\Omega_r\to \R$ such that 
\begin{empheq}[left=\empheqlbrace]{align}
\label{eq:Num1}
-\MVPrh u_h(x) =f(x), &\quad  x\in \Omega,\\
\label{eq:Num2}u_h(x)=G(x), &\quad  x\in \partial \Omega_r.
\end{empheq}
We have the following result.
\begin{theorem}\label{thm:ConvNumSch}
\label{thm:dpp}
Let $p\in(1,\infty)$, $\Omega\subset \R^d$ be a bounded, open and $C^2$ domain,  $f\in C(\overline{\Omega})$ and $g\in C(\partial \Omega)$.  Assume \eqref{eq:ass}.
\begin{enumerate}[\rm (a)]
\item\label{thm:ConvNumSch-item1} Then there exists a unique pointwise solution $u_h\in L^\infty(\Omega_r)$ of \eqref{eq:Num1}-\eqref{eq:Num2} when $r$ is small enough.
\item\label{thm:ConvNumSch-item2} If $u$ is the unique viscosity solution of \eqref{eq:BVP}-\eqref{eq:BC}, then
\[
 \sup_{x\in \overline{\Omega}}\left|u_h(x)-u(x)\right| \to 0 \quad \textup{as} \quad r\to0^+.
\]
\end{enumerate} 
\end{theorem}
\begin{remark}\label{rem:conj}
 We conjecture that the relation $h=o(r^{3/2})$ is sufficient also in the range $p\in(1,3)$. See Section \ref{sec:fullynon} for numerical evidence supporting this.
\end{remark}

We note that if we restrict \eqref{eq:Num1}-\eqref{eq:Num2} to the uniform grid $\Grid$ we obtain a fully discrete problem suited for numerical computations. More precisely, define the discrete sets
\[
B_r^h:=B_r\cap \Grid,  \quad \Omega^h:= \Omega\cap \Grid, \quad \partial\Omega^h:= \partial\Omega\cap \Grid \quad \textup{and} \quad \Omega^h_r:= \Omega_r\cap \Grid.
\]
Observe that  $\MVPrh$ given in \eqref{eq:discPlap} can be interpreted as an operator $\MVPrh:\ell^\infty(\Grid)\to \ell^\infty(\Grid)$ since given any $x_\beta, y_\alpha\in \Grid $ we have $x_\beta+y_\alpha=(\beta+\alpha)h = x_{\beta+\alpha}\in \Grid$ and then
\[
\MVPrh\phi_\beta:= \frac{h^d}{D_{d,p}\, \omega_d\, r^{p+d}} \sum_{y_\alpha\in B_r} J_p(\phi_{\beta+\alpha}-\phi_\beta) \quad \textup{for}\quad x_\beta \in \Omega^h 
\]
with $\phi: \Grid\to \R$ and $\phi_{\gamma}:= \phi(\gamma h)$, whenever  $\gamma h \in \Grid$. Finally note that if $x_\beta\in \Omega^h$ and $y_\alpha \in B_r^h$ we have that $x_\beta+y_\alpha= x_{\beta+\alpha} \in \Omega_r^h$, so that \eqref{eq:Num1}-\eqref{eq:Num2} can be interpreted as 
\begin{empheq}[left=\empheqlbrace]{align}
\label{eq:NumDisc1}
-\MVPrh U_\beta =f_\beta, &\quad  x_\beta\in \Omega^h\\
\label{eq:NumDisc2}\ U_\beta=G_\beta, &\quad  x_\beta\in \partial \Omega_r^h,
\end{empheq}
with $U:\Omega_r^h\to \R$, $f_\beta:=f(x_\beta)$ and $G_\beta:=G(x_\beta)$. In this way we have the following trivial consequence of Theorem  \ref{thm:ConvNumSch}. 
\begin{corollary}\label{coro:ConvNumSch}
\label{cor:dpp}
 Assume the hypotheses of Theorem \ref{thm:ConvNumSch}.
\begin{enumerate}[\rm (a)]
\item Then there exists a unique pointwise solution $U\in \ell^\infty(\Omega_r^h)$ of \eqref{eq:NumDisc1}-\eqref{eq:NumDisc2} when $r$ is small enough.
\item If $u$ is the unique viscosity solution of \eqref{eq:BVP}-\eqref{eq:BC}, then
\[
 \max_{x_\beta\in \Omega_h}\left|U_\beta-u(x_\beta)\right| \to 0 \quad \textup{as} \quad r\to0.
\]

\end{enumerate} 
\end{corollary}

\subsection{Related results} \label{sec:relres}
 
 For an overview of classical and modern results for the $p$-Laplacian, we refer the reader to the book \cite{Linq19}. For an overview of numerical methods for degenerate elliptic PDEs we refer the reader to Section 1.1 in \cite{ObermanpLap}.
 
We want to stress that the operator of interest in this paper is the \emph{variational} $p$-Laplacian, i.e.,
\[
\lap_p u = \Div(|\nabla u|^{p-2}\nabla u). 
\]
%To our best knowledge, finite difference discretizations of this operator have never been presented in the literature, and this is, in our opinion, the main novelty of this work. \normalcolor

On the other hand, finite difference methods for equations involving the $p$-Laplacian have been successfully developed using the \emph{normalized} (or game theoretical) version of the $p$-Laplacian  $\lap_p^\textup{N}$. The ideas are based on the identity
$$
\lap_p u =|\nabla u|^{p-2}\lap u +(p-2)|\nabla u|^{p-4}\lap_\infty u.
$$
This allows to define
$$
\lap_p^\textup{N} u :=  |\nabla u|^{2-p}\lap_p u = \lap u +(p-2)\lap_\infty^\textup{N} u,
$$
where $\lap_\infty^\textup{N}$ is the so-called normalized infinity Laplacian, which is given by the second order directional derivative in the direction of the gradient.  One limitation of such methods is the fact that they are not adapted to treat nonhomogeneous problems of the form $-\lap_p u=f$. Instead they allow for treating inhomogeneities of the form $-\lap_p u=|\nabla u|^{p-2}f$ (both problems are equivalent only if $f\equiv0$). 

Let us first comment on the literature related to finite difference methods for $\lap_p^\textup{N}$. In \cite{ObermanpLap}, the author presents a monotone finite difference scheme for the normalized infinity Laplacian and the game theoretical (or normalized) $p$-Laplacian for $p\geq2$. In addition, a scheme for  \eqref{eq:BVP}-\eqref{eq:BC} with $f\equiv0$ is presented, 
 together with a semi-implicit solver. In \cite{dTMP18}, a strategy to prove the convergence of dynamic programming principles (including monotone finite difference schemes) for the normalized $p$-Laplacian is presented,  as well as the strong uniqueness property for the $p$-Laplacian, which is crucial for the application of the convergence criteria of Barles and Souganidis in \cite{BS91}. We also seize the opportunity mention Section 6 in \cite{CoLeMa17}, where a finite difference method (based on the mean value properties of the normalized $p$-Laplacian) is proposed for a double-obstacle problem involving the $p$-Laplacian. We note that in the case $1<p<2$ neither of the above mentioned schemes are monotone, and as such, the numerical scheme in this paper is the first one treating this range, even in the homogeneous case $f\equiv0$.

There are many other monotone approximations of $\lap_p^\textup{N}$ available in the literature. Strictly speaking, they are not numerical approximations, but the proof of convergence follows similar strategies based on monotonicity and consistency. See \cite{dTMP18} for a discussion on this topic. Such approximations were first presented in \cite{MRP} (see also  \cite{LewMan14,MPR12a, MPR12b} for a probabilistic game theoretical approach). The basic idea of these approximations is to combine the classical mean value property (MVP) for the Laplacian with a MVP for the normalized infinity Laplacian motivated by Tug-of-War games \cite{Pe09}. The literature on this topic has become extensive in the last decade. In \cite{AL16, LM16} the equivalence between being $p$-harmonic and satisfying a MVP is treated. See \cite{KMP12,Le20} for a MVP in the full range $1<p<\infty$ and \cite{LeMa17} for the application of such approximations in the context of obstacle problems.

Regarding monotone approximations of the variational $p$-Laplacian, the literature is very recent and not so extensive. The MVP given by \eqref{eq:MVF} was derived in \cite{BS18, dTLi20}. In \cite{dTLi20} it is shown to be a monotone approximation of $\Delta_p$. The authors are also able to prove convergence of the corresponding approximating problems to a viscosity solution. 

It is noteworthy that the discretization presented in this paper is reminiscent of the definition of the variational $p$-Laplacian on graphs, see \cite{Amb03} and also \cite{Yam79}. In this direction, Corollary \ref{coro:ConvNumSch} can be interpreted as the convergence of the solution to a PDE defined on a graph associated to the grid. We refer to the recent paper \cite{Wang21} for a study of the eigenvalues of this operator and to \cite{Abd15} for its applications to image processing. Note that also the normalized $p$-Laplacian has been defined on graphs, see \cite{MOS15}.

Finally, we seize the opportunity to mention that since the $p$-Laplacian is of divergence form, it is well suited for finite element based methods. We mention a few papers in this direction:  \cite{BL93}, \cite{FedPP-L}, \cite{GM75}, \cite{Kwa05}, \cite{WN01}, \cite{WN02} \cite{LB93} and \cite{Loi20}. We want to stress that finite element methods does not produce monotone approximations, and thus, 
are not well suited for treating viscosity solutions.
\subsection{Organization of the paper}

In Section \ref{sec:not}, we introduce some notation and prerequisites needed in the rest of the paper.  Section \ref{sec:cons} is devoted to the proof of consistency of the discretization previously introduced. In Section \ref{sec:scheme}, we study the numerical scheme for the boundary value problems. 
This is followed by  a discussion around solving the nonlinear systems of equations derived from our scheme, in Section \ref{sec:NonlinearSis}. Finally, in Section \ref{sec:exp}, we perform some numerical experiments to support our theoretical results. We also have an appendix containing technical results.

\section{Notations and prerequisites}\label{sec:not}
We adopt the following definition of viscosity solutions, which is the classical definition adjusted to the nonhomogeneous equation (see e.g. \cite{equivalence2}). 
 \begin{definition} [Solutions of the equation] \label{pvisc} Suppose that $f\in C(\Omega)$. We say that a lower (resp. upper) semicontinuous function $u$ in $\Omega$ is a \emph{viscosity supersolution} (resp. \emph{subsolution})  of the
  equation $$-\plap u=f $$  in $\Omega$ if the
  following holds: whenever $x_0 \in \Omega$ and $\varphi \in
  C^2(B_R(x_0))$ for some $R>0$ are such that $|\nabla \varphi(x)|\neq 0$ for $x\in B_R(x_0)\setminus\{x_0\}$, 
$$ \varphi(x_0) = u(x_0) \quad \text{and} \quad \varphi(x) \leq
u(x) \quad  \text{(resp. $\varphi(x)\geq u(x)$)} \quad \text{for all} \quad
x \in  B_R(x_0)\cap \Omega,$$
then we have
\begin{equation}\label{eq:testfuncdef}\lim_{\rho\to 0}\sup_{B_\rho(x_0)\setminus\{x_0\}}\left(-\plap \varphi (x)\right)\geq f(x_0) \quad   \text{(resp. $\lim_{\rho\to 0}\inf_{B_\rho(x_0)\setminus\{x_0\}}\left(-\plap \varphi (x)\right)\leq f(x_0)$)}.
\end{equation}
 A \emph{viscosity solution} is a function $u\in C(\Omega)$ being
both a viscosity supersolution and a viscosity subsolution.
\end{definition}
 
\begin{remark}\label{rem:equivviscform}
We consider condition \eqref{eq:testfuncdef} to avoid problems with the definition of $-\Delta_p \varphi(x_0)$ when $|\nabla \varphi(x_0)|=0$ and $p\in(1,2)$. However, when either $p\geq2$ or $|\nabla \varphi(x_0)|\not=0$, \eqref{eq:testfuncdef} can be replaced by the standard one, i.e.,
\begin{equation}\label{eq:testfuncdef2}
-\plap \varphi (x_0)\geq f(x_0) \quad   \text{(resp. $-\plap \varphi (x_0)\leq f(x_0)$)}.
\end{equation}
\end{remark}
A viscosity solution of the boundary value problem \eqref{eq:BVP}-\eqref{eq:BC} attaining the boundary condition in a pointwise sense is naturally defined as follows.
\begin{definition}[Solutions of the boundary value problem] 
Suppose that $f\in C(\overline{\Omega})$ and $g\in C(\partial \Omega)$. We say that a lower (resp. upper) semicontinuous function $u$  in $\overline{\Omega}$  is a \emph{viscosity supersolution} (resp. \emph{subsolution}) of \eqref{eq:BVP}-\eqref{eq:BC} if
\begin{enumerate}[(a)]
\item $u$ is a viscosity supersolution (resp. subsolution) of $-\Delta_p u=f$ in $\Omega$ (as in Definition \ref{pvisc});
\item $u(x)\geq g(x)$ (resp. $u(x)\leq g(x)$) for $x\in \partial \Omega.$
\end{enumerate}
A \emph{viscosity solution} of \eqref{eq:BVP}-\eqref{eq:BC} is a function $u\in C(\overline{\Omega})$ being both a viscosity supersolution and a viscosity subsolution.
\end{definition}

\begin{remark}
To prove the convergence result (Theorem \ref{thm:ConvNumSch}\eqref{thm:ConvNumSch-item2}) we will make use of a generalized notion of viscosity solutions of a boundary value problem. We will introduce this notion just before using it. See Section \ref{sec:conv}.
\end{remark}

\section{Consistency of the discretization: Proof of Theorem \ref{thm:MVF}}\label{sec:cons}
In this section we prove the consistency of the discretization $\MVPrh$ for $C^2$-functions as presented in Theorem \ref{thm:MVF}.
\begin{proof}[Proof of Theorem \ref{thm:MVF}]  Throughout this proof, $C$ will denote a constant that may depend on $p$, the dimension $d$, but not on $r$ or $h$.

The mean value property introduced in \cite{dTLi20} involves the quantity
\[
\Mmvp[\phi](x)=\frac{1}{D_{d,p} r^p} \fint_{ B_r} J_p(\phi(x+y)-\phi(x)) \dd y.
\]
By the triangle inequality and Theorem 2.1 in \cite{dTLi20}
\[
\begin{split}
\left|\MVPrh\phi(x)-\lap_p \phi(x)\right|&\leq \left|\MVPrh\phi(x)-\Mmvp[\phi](x)\right|+ \left|\Mmvp[\phi](x)-\lap_p \phi(x)\right|\\
&= \left|\MVPrh\phi(x)-\Mmvp[\phi](x)\right|+ o_r(1) \qquad \textup{as} \qquad r\to0^+.
\end{split}
\]
Therefore, it is sufficient to show that
 $$|\MVPrh\phi(x)-\Mmvp[\phi](x)|=o_r(1)  \qquad \textup{as} \qquad r\to0^+.$$
\emph{Step 1: Approximation of $B_r$ by $h$-boxes.}
Define the following family of $h$-boxes centred at $y_\alpha\in \Grid$,
\[
R^h_\alpha:= y_\alpha + \frac{h}{2}[-1,1)^d,
\]
and the union of boxes that approximates $B_r$
\[
 \tilde{B}_r:= \bigcup_{y_\alpha \in B_r^h} R^h_\alpha.
\]
See Figure \ref{fig:bc1}. 
\vspace{-1cm}
\begin{figure}[h]
\includegraphics[scale=0.7]{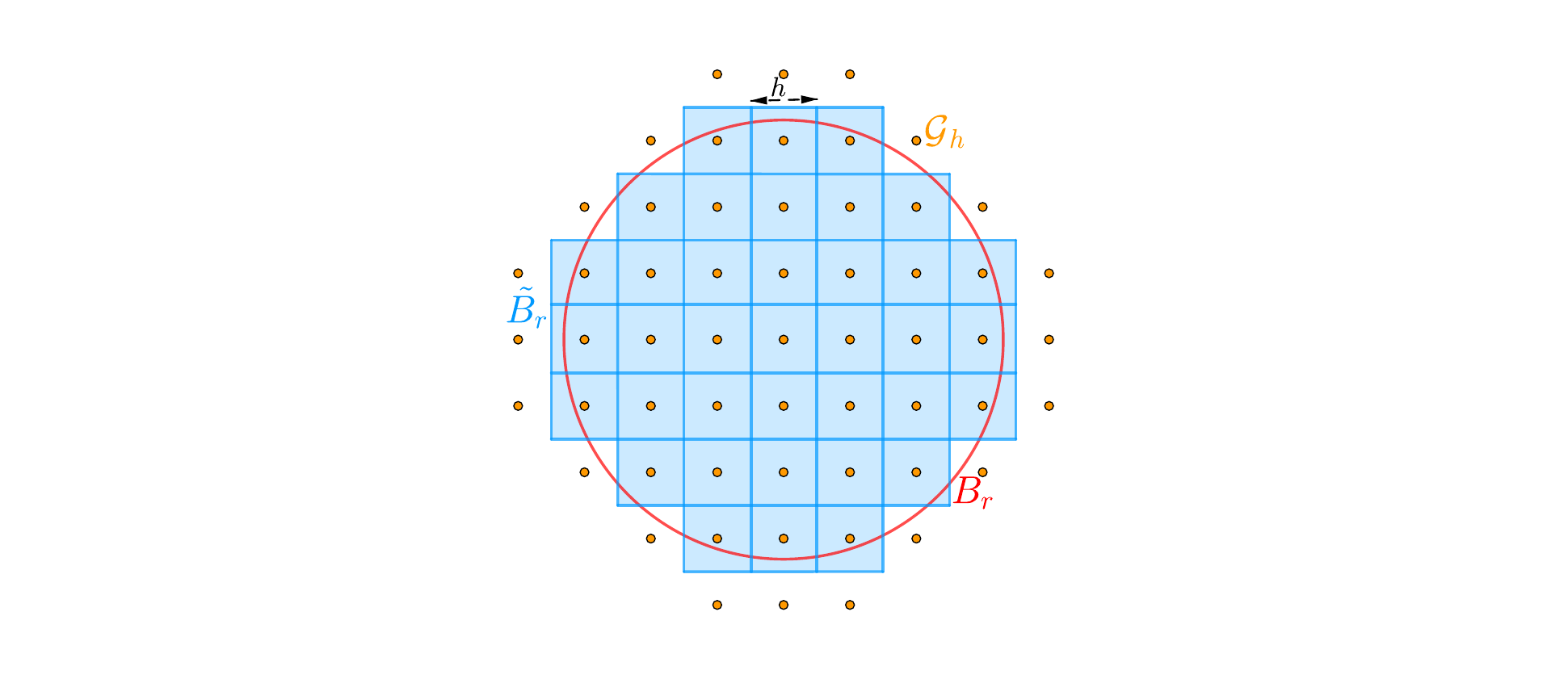}
\vspace{-0.8cm}
\caption{The family of boxes and their union $\tilde B_r$ that covers $B_r$.}
\label{fig:bc1}
\end{figure}

Consider 
\[
A_r:=\frac{1}{2}\int_{ B_r} \left(J_p(\phi(x+y)-\phi(x))+J_p(\phi(x-y)-\phi(x))\right) \dd y,
\] 
and
\[
\tilde{A}_r:=\frac{1}{2}\int_{ \tilde{B}_r} \left(J_p(\phi(x+y)-\phi(x))+J_p(\phi(x-y)-\phi(x))\right)  \dd y.
\]
In this step we will prove that 
\begin{equation}\label{eq:Ar}
|A_r-\tilde{A}_r|=o(r^{d+p}).
\end{equation}
Notice first that
\[
\begin{split}
\left|A_r-\tilde{A}_r\right|= \frac{1}{2}&\left| \int_{ B_r\setminus \tilde{B_r}} \left(J_p(\phi(x+y)-\phi(x))+J_p(\phi(x-y)-\phi(x))\right) \dd y  \right.\\
&\qquad  \left.- \int_{ \tilde{B}_r\setminus B_r} \left(J_p(\phi(x+y)-\phi(x))+J_p(\phi(x-y)-\phi(x))\right) \dd y  \right|\\
\leq \quad & \frac12\int_{(B_r\cup\tilde{B}_r)\setminus (B_r\cap\tilde{B}_r)} \left|J_p(\phi(x+y)-\phi(x))+J_p(\phi(x-y)-\phi(x))\right| \dd y.
\end{split} 
\]
It is easy to verify that $B_r\cup\tilde{B}_r\subset B_{r+\sqrt{d}h}$ and $B_{r-\sqrt{d}h}\subset B_r\cap\tilde{B}_r$ so that 
\[
(B_r\cup\tilde{B}_r)\setminus (B_r\cap\tilde{B}_r)\subset B_{r+\sqrt{d}h}\setminus B_{r-\sqrt{d}h}.
\]
Observe that regardless of the value of $p$, we always have $h=o(r)$. Therefore, 
\[
\begin{split}
|(B_r\cup\tilde{B}_r)\setminus (B_r\cap\tilde{B}_r)|&\leq |B_{r+\sqrt{d}h}\setminus B_{r-\sqrt{d}h}|= \omega_d \left((r+\sqrt{d}h)^d- (r-\sqrt{d}h)^d\right)\\
&\leq  \omega_d d(r+\sqrt{d}h)^{d-1}2\sqrt{d}h\leq  C r^{d-1}h\\
&= o(r^{d}).
\end{split}
\]
On the other hand, by Taylor expansion
\[
|\phi(x+y)-\phi(x)+ \phi(x-y)-\phi(x)|=\mathcal{O}(|y|^2).
\]
 In the case $p\geq2$, Lemma \ref{lem:pineq1} implies
\[
\begin{split}
|J_p&(\phi(x+y)-\phi(x))+J_p(\phi(x-y)-\phi(x))|= |J_p(\phi(x+y)-\phi(x))-J_p(-\phi(x-y)+\phi(x))|\\
&\leq C \max( |\phi(x-y)-\phi(x)|,|\phi(x+y)-\phi(x)|)^{p-2}|\phi(x+y)-\phi(x)+ \phi(x-y)-\phi(x)|=O(|y|^p).
\end{split}
\]
We can conclude 
\[
\begin{split}
|A_r-\tilde{A}_{r}|
&\leq \frac{\tilde{C}}{2} \int_{(B_r\cup\tilde{B}_r)\setminus (B_r\cap\tilde{B}_r)} |y|^{p} \dd y\leq \frac{1}{2}(r+\sqrt{d}h)^{p} |(B_r\cup\tilde{B}_r)\setminus (B_r\cap\tilde{B}_r)|= o(r^{p+d}).
\end{split} 
\]
In the case $p< 2$, we argue slightly different. On page 8 in \cite{dTLi20} it is proved that 
$$
\fint_{\partial B_r}|J_p(\phi(x+y)-\phi(x))-J_p(y\cdot \nabla \phi(x)+\frac12y^T D^2\phi(x) y)| dy =o(r^p).
$$
In a similar fashion, one can prove
\[
\begin{split}
\int_{(B_r\cup\tilde{B}_r)\setminus (B_r\cap\tilde{B}_r)}\Big|&J_p(\phi(x+y)-\phi(x))+J_p(\phi(x-y)-\phi(x))\\&-J_p(y\cdot \nabla \phi(x)+\frac12y^T D^2\phi(x) y)-J_p(-y\cdot \nabla \phi(x)+\frac12y^T D^2\phi(x) y)\Big| dy =o(r^{p+d}).
\end{split}
\]
To show \eqref{eq:Ar} it is therefore sufficient to show that 
\begin{equation}\label{eq:eq1}
\int_{(B_r\cup\tilde{B}_r)\setminus (B_r\cap\tilde{B}_r)}\left|J_p(y\cdot \nabla \phi(x)+\frac12y^T D^2\phi(x) y)+J_p(-y\cdot \nabla \phi(x)+\frac12y^T D^2\phi(x) y)\right| dy=o(r^{p+d}).
\end{equation}
Without loss of generality assume that $\nabla \phi(x)=ce_1$ with $c\not=0$. Then
\begin{equation*}
\begin{split}
J_p(y\cdot \nabla\phi (x) +\frac{1}{2}y^T D^2 \phi(x) y)&=J_p(c y \cdot \vec{e}_1 +\frac12 y^T D^2 \phi(x) y)=(c|y|)^{p-1} J_p(\hat{y} \cdot \vec{e}_1 +\frac12 c^{-1}|y|\hat{y}^T D^2 \phi(x) \hat{y} )
\end{split}
\end{equation*}
where $\hat{y}=y/|y|$. By Lemma \ref{lem:pineq3} with $a=\hat{y} \cdot \vec{e}_1$ and $b=\frac12  c^{-1}|y|\hat{y}^T D^2 \phi(x) \hat{y}$ we get
 \[
\begin{split}
(c|y|)^{p-1} \Big|J_p(\hat{y} \cdot \vec{e}_1 +&\frac12 c^{-1}|y|\hat{y}^T D^2 \phi(x) \hat{y} )- J_p(\hat{y} \cdot \vec{e}_1)\Big|\\
&\leq C(c|y|)^{p-1}\left(|\hat{y} \cdot \vec{e}_1|+\frac12  c^{-1}|y||\hat{y}^T D^2 \phi(x) \hat{y}|\right)^{p-2}\frac12 c^{-1}|y||\hat{y}^T D^2 \phi(x) \hat{y}|\\
&\leq C|y|^p|\hat{y} \cdot \vec{e}_1|^{p-2}.
\end{split}
\]
Hence, 
$$
\left|J_p(y\cdot \nabla \phi(x)+\frac12y^T D^2\phi(x) y)+J_p(-y\cdot \nabla \phi(x)+\frac12y^T D^2\phi(x) y)\right|\leq C|y|^p|\hat{y} \cdot \vec{e}_1|^{p-2}.
$$
From (6.2) in \cite{dTLi20} it then follows that 
\begin{equation*}
\begin{split}
\int_{\partial B_r}\Big|J_p(y\cdot \nabla \phi(x)+\frac12y^T D^2\phi(x) y)+J_p(-y\cdot \nabla \phi(x)+&\frac12y^T D^2\phi(x) y)\Big| d\sigma(y) \\
&\leq C r^p \int_{\partial B_r}|\hat{y} \cdot \vec{e}_1|^{p-2} \dd\sigma(y)\leq \tilde{C}r^{p+d-1}.
\end{split}
\end{equation*}
After integration (to  pass from spheres to balls) we obtain 
\[
\begin{split}
\int_{(B_r\cup\tilde{B}_r)\setminus (B_r\cap\tilde{B}_r)}&\left|J_p(y\cdot \nabla \phi(x)+\frac12y^T D^2\phi(x) y)+J_p(-y\cdot \nabla \phi(x)+\frac12y^T D^2\phi(x) y)\right| dy \\
&\leq \int_{B_{r+\sqrt{d}h}\setminus B_{r-\sqrt{d}h}}\left|J_p(y\cdot \nabla \phi(x)+\frac12y^T D^2\phi(x) y)+J_p(-y\cdot \nabla \phi(x)+\frac12y^T D^2\phi(x) y)\right| dy \\
& \leq Chr^{p+d-1}=o(r^{d+p}).
\end{split}
\]
This is \eqref{eq:eq1}.

\emph{Step 2: Discretization of $\tilde{A}_r$. } Consider 
\[
\tilde{A}_r^h:=h^d\sum_{y_\alpha\in B_r}J_p(\phi(x+y_\alpha)-\phi(x)).
\]
We will show that
\begin{equation}\label{eq:Arh}
|\tilde{A}_r-\tilde{A}_r^h|=o(r^{d+p}).
\end{equation}

Observe that 
\[
\begin{split}
\tilde{A}_r&=\frac{1}{2}\sum_{y_\alpha\in B_r}\int_{ R_\alpha^h} \left(J_p(\phi(x+y)-\phi(x))+J_p(\phi(x-y)-\phi(x))\right)  \dd y\\
&=\frac{1}{2}\sum_{y_\alpha\in B_r}\int_{ R_0^h} \left(J_p(\phi(x+y_\alpha+y)-\phi(x))+J_p(\phi(x+y_\alpha-y)-\phi(x))\right)  \dd y.
\end{split}
\]
Since $|R_0^h|=h^d$ we have
\[
\begin{split}
|\tilde{A}_{r}&-A_r^h|\\
&= \frac{1}{2}\left| \sum_{y_\alpha\in B_r} \int_{R_0^h} \left(J_p(\phi(x+y_\alpha+y)-\phi(x))+J_p(\phi(x+y_\alpha-y)-\phi(x)) - 2J_p(\phi(x+y_\alpha)-\phi(x)) \right)\dd y \right|.
\end{split} 
\]
If $p\geq2$ we use Taylor expansion of order two and obtain
$$
\phi(x+y_\alpha\pm y)-\phi(x)=\phi(x+y_\alpha)-\phi(x)\pm \nabla \phi(x+y_\alpha)\cdot y + \mathcal{O}(y^2).
$$
Let $\rho=\phi(x+y_\alpha)-\phi(x)$ and $\eta=\nabla\phi(x+y_\alpha)$. Then this can be expressed as
$$
\phi(x+y_\alpha+y)-\phi(x)=\rho+\eta\cdot y+\mathcal{O}(|y|^2).
$$
Therefore, by Lemma \ref{lem:pineq1}  
\[
\begin{split}
\Big|J_p(\phi(x+y_\alpha+y)-\phi(x))-J_p(\rho+\eta\cdot y)\Big|&\leq C\max (|\phi(x+y_\alpha+y)-\phi(x)|,|\rho+\eta\cdot y|)^{p-2}|y|^2\\
&\leq Cr^{p-2}o(r^2)\\
&= o(r^{p}),
\end{split}
\]
where we have used that $y=\mathcal{O}(h)=o(r)$ and that $\rho=\mathcal{O}(y_\alpha)=\mathcal{O}(r)$.
It follows that it will be enough to obtain an estimate of the form 
\begin{equation}
\label{eq:long}
|J_p(\rho+\eta\cdot y)+J_p(\rho-\eta \cdot y)-2J_p(\rho)| = o(r^p).
\end{equation}
For $p=2$, this estimate is trivial. When $p >3$ we use the second order Taylor expansion of $J_p$ to obtain 
\begin{equation}
\label{eq:pgeq3}
\begin{split}
|J_p(\rho+\eta\cdot y)- J_p(\rho)-(p-1)|\rho|^{p-2}\eta\cdot y|&\leq C\max(|\rho|,|\rho+\eta\cdot y|)^{p-3}|\eta\cdot y|^{2}\\
&\leq Cr^{p-3}o(r^{3}) =o(r^p), 
\end{split}
\end{equation}
since $\rho=\mathcal{O}(y_\alpha)=\mathcal{O}(r)$ and $y=\mathcal{O}(h)=o(r^\frac32)$ when $p>3$.

When $p\in (2,3]$ we use the fact that the derivative of the function $t\mapsto J_p(t)$ is $({p-2})$-H\"older continuous and obtain 
\begin{equation}
\label{eq:p2to3}
\begin{split}
|J_p(\rho+\eta\cdot y)- J_p(\rho)-(p-1)|\rho|^{p-2}\eta\cdot y|&\leq C|\eta\cdot y|^{p-1}\\
& =o(r^p), 
\end{split}
\end{equation}
where we used that $y=\mathcal{O}(h)=o(r^{p/(p-1)})$ when $p\in (2,3]$. The estimate \eqref{eq:long} follows immediately from \eqref{eq:pgeq3} and \eqref{eq:p2to3}, respectively.

If $p<2$ we use the fact that $J_p$ is $(p-1)$-H\" older continuous. Thus,
\begin{equation}
\begin{split}\label{eq:long2}
\big|J_p(\phi(x+y_\alpha+y)-\phi(x))-J_p(\phi(x+y_\alpha)-\phi(x) )\big|&\leq C|\phi(x+y_\alpha+y)-\phi(x+y_\alpha)|^{p-1}
\\&\leq C|y|^{p-1}=o(r^p),
\end{split}
\end{equation}
where we used the assumption $y=\mathcal{O}(h)=o(r^{p/(p-1)})$ when $p<2$. Using \eqref{eq:long} and \eqref{eq:long2} we get
\[
|\tilde{A}_{r}-A_r^h|=o(r^p)\sum_{y_\alpha\in B_r} h^d=o(r^p)|\tilde{B}_r|\leq o(r^p) |B_{r+\sqrt{d}h}|=o(r^{p+d}).
\]
\emph{Step 3: Conclusion.} Combining Step 1 and Step 2, we obtain
\[
\begin{split}
|\MVPrh\phi(x)-\Mmvp[\phi](x)|&=\frac{1}{D_{d,p}r^p |B_r|}|A_r^h-A_r|\\
&\leq \frac{C}{ r^{p+d}} \left(|A_r^h-\tilde{A}_r|+ |\tilde{A}_r-A_r|\right)= \frac{1}{ r^{p+d}} o(r^{p+d})= o_r(1).\qedhere
\end{split}
\]
\end{proof}

\section{Properties of the numerical scheme}\label{sec:scheme}
In this section we will state and prove some properties of the numerical scheme \eqref{eq:Num1}-\eqref{eq:Num2}.

\subsection{Existence and uniqueness}\label{sec:existanduniquenumsol} We will obtain the existence and uniqueness result given in Theorem \ref{thm:ConvNumSch}\eqref{thm:ConvNumSch-item1}.

First note that we can write 
\[
\begin{split}
\MVPrh\phi(x)&= \frac{h^d}{D_{d,p}\, \omega_d\, r^{p+d}} \sum_{y_\alpha\in B_r} J_p(\phi(x+y_\alpha)-\phi(x))\\
&=\frac{1}{D_{d,p}r^p} \fint_{B_r}J_p(\phi(x+y)-\phi(x))\dd \mu(y)
\end{split}
\]
with $\mu$ being the discrete measure given by
\[
\dd \mu(y):=h^d\sum_{y_\alpha\in B_r} \dd \delta_{y_\alpha}(y),
\]
where $\delta_{z}$ denotes the dirac delta measure at $z\in \R^d$. With this simple observation, all the results of Section 9.1 in \cite{dTLi20} follow here word by word (replacing $\mathcal{M}^p_{r}$ by $\MVPrh$ and $\dd y$ by $\dd \mu(y)$). We state them for completeness. Our running assumptions in this section will be $f\in C(\overline{\Omega})$ and $G\in C(\partial \Omega_r)$ (a continuous extension of $g\in C(\partial \Omega)$).

The comparison result below implies in particular the uniqueness of solutions of  \eqref{eq:Num1}-\eqref{eq:Num2}. 
\begin{proposition}[Comparison]\label{prop:comparison}
Let $p\in (1,\infty)$, $h,r>0$, and $v,w\in L^\infty(\Omega_r)$ be such that
\[
\begin{cases}
-\MVPrh w (x)\geq f(x),& x\in \Omega,\\
\qquad w(x)\geq G(x),& x\in \partial \Omega_r,
\end{cases} \qquad 
\textup{and}
\qquad 
 \begin{cases}
-\MVPrh v (x)\leq f(x),& x\in \Omega,\\
\qquad v(x)\leq G(x),& x\in \partial \Omega_r.
\end{cases}
\]
Then $u\leq w$ in $\Omega_r$.
\end{proposition}

The existence of solutions is proved by a monotonicity argument. For this purpose, we need the following $L^\infty$-bound. 

\begin{proposition}[$L^\infty$-bound]\label{Prop:Stability}
Let $p\in (1,\infty)$, let $R>0$ and $u_h$ be the solution (if any) of \eqref{eq:Num1}-\eqref{eq:Num2} corresponding to some $r\leq R$.  Assume \eqref{eq:ass}. Then 
\[
\|u_h\|_{\infty}\leq A,
\]
 for $r$ small enough, with $A>0$ depending on $p, \Omega, f, g$ and $R$ (but not on $r$ and $h$).
\end{proposition}
\begin{proof}
See the proof of Proposition 9.2 \cite{dTLi20}.  The proof is based on an explicit barrier for the $p$-Laplace equation, which by Theorem \ref{thm:MVF} gives a  barrier for \eqref{eq:Num1}-\eqref{eq:Num2}. 
\end{proof}
In order to prove the existence we also need a two step iteration process. For that purpose we define
\[
L[\psi,\phi](x):= \frac{1}{D_{d,p}r^p} \fint_{ B_r} J_p(\phi(x+y)-\psi(x) )\dd \mu(y).
\]
We have the following result.
\begin{lemma}\label{lemma:L}
Let $r>0$
and $\phi\in L^\infty(\Omega_r)$.
\begin{enumerate}[\rm (a)]
\item\label{lemma:L-item1} Then there exists a unique $\psi\in L^\infty(\Omega)$ such that
$-L[\psi,\phi](x)=f(x)$ for all $x\in \Omega$.
\item\label{lemma:L-item2}  Let $\psi_1$ and $\psi_2$ be such that
$
- L[\psi_1,\phi](x) \leq  f(x)$ and $- L[\psi_2,\phi](x)\geq f(x)$ for all $x\in \Omega
$,
then $\psi_1\leq \psi_2$ in $\Omega$.
\end{enumerate}
\end{lemma}

\begin{proof}
The proof follows as the proof of Lemma 9.3 in \cite{dTLi20}.
\end{proof}
We are finally ready to prove the existence.
\begin{proof}[Proof of  Theorem \ref{thm:ConvNumSch}\eqref{thm:ConvNumSch-item1}]
The proof follows the proof of Proposition 9.4 in \cite{dTLi20}. We spell out some details below. 

The approach for existence is to construct a monotone increasing sequence converging to the solution. Let $\mathcal{B}$ be the barrier constructed in Proposition \ref{Prop:Stability}. Define
\[
u^0_h(x)=
\begin{cases}
\displaystyle\inf_{ \partial \Omega_r} G -\mathcal{B}(x) &x\in \Omega,\\
G(x) &x\in \partial\Omega_r,
\end{cases}
\]
and the sequence $u^k_h$ as the sequence of solutions of 
\[
\begin{cases}
 -  L[u^k_h,u^{k-1}_h](x)= f(x)& x\in \Omega,\\
\qquad \qquad \, \, \, \, u^k_h(x)= G(x)& x\in \partial \Omega_r.
\end{cases}
\]
One can prove that $u^k_h$ exists for all $k$, is nondecreasing (by the monotonicity of $L$) and uniformly bounded (by Proposition \ref{Prop:Stability}). We can then define the pointwise limit
$$
u_h(x):=\lim_{k\to \infty} u^k_h(x).
$$
Due to the the pointwise  convergence
\[
\begin{split}
-f(x)&=\lim_{k\to \infty} L[u^{k+1}_h, u^k_h](x)=  L[\lim_{k\to \infty} u^{k+1}_h, \lim_{k\to \infty} u^k_h](x)=L[u_h, u_h](x)= \MVPrh[u_h](x).
\end{split}
\]
Thus, $u$ is a solution of \eqref{eq:Num1}. Clearly $u_h=G$ in $\partial \Omega_r$ so it is also a solution of  \eqref{eq:Num2}.  The uniqueness follows from Proposition \ref{prop:comparison}.
\end{proof}

\subsection{Monotonicity and consistency} In order to prove convergence of the numerical scheme, we will  need certain monotonicity and consistency properties (we already obtained a uniform bound in Proposition \ref{Prop:Stability}). For a function $\phi:\Omega_r\to\R$ define 
\[
S(r,h,x,\phi(x),\phi):=\begin{cases}
\displaystyle - \frac{h^d}{D_{d,p}\, \omega_d\, r^{p+d}} \sum_{y_\alpha\in B_r} J_p(\phi(x+y_\alpha)-\phi(x))- f(x)  &x\in \Omega,\\
\phi(x)-G(x) &x\in \partial\Omega_r.
\end{cases}
\]
Note that \eqref{eq:Num1}-\eqref{eq:Num2} can be equivalently formulated as
\[
S(r,h,x,u_h(x),u_h)=0 \quad x\in \Omega_r.
\]
We have the following result.
\begin{lemma}\label{lem:propscheme}
Assume \eqref{eq:ass}.
\begin{enumerate}[\rm(a)]
\item\label{lem:propscheme-item-a} \textup{(Monotonicity)} Let $t\in\R$ and $\psi\geq\phi$. Then
\[
S(r,h,x,t,\psi)\leq S(r,h,x,t,\phi)
\]
\item\label{lem:propscheme-item-b}  \textup{(Consistency)}  For all $x\in \overline{\Omega}$ and  $\phi\in C^2( B_R(x))$ for some $R>0$  such that $|\nabla \phi(x)|\neq 0$ we have that
\[
\begin{split}
\limsup_{r\to0, z\to x, \xi \to 0} S(r,h, z&, \phi(z)+\xi+\eta_{r}, \phi+\xi)\leq \left\{\begin{array}{cccl}
-\Delta_p\phi(x)-f(x)& \text{if} &x\in\Omega\\
\max\{-\Delta_p\phi(x)-f(x), \phi(x)-g(x)\}& \text{if} &x\in \partial{\Omega},
\end{array}\right.
\end{split}
\]
\noindent and 
\[
\begin{split}
\liminf_{r\to0, z\to x, \xi \to 0} S(r,h, z&, \phi(z)+\xi-\eta_{r}, \phi+\xi)\geq \left\{\begin{array}{cccl}
-\Delta_p\phi(x)-f(x)& \text{if} &x\in\Omega\\
\min\{-\Delta_p\phi(x)-f(x), \phi(x)-g(x)\}& \text{if} &x\in \partial{\Omega},
\end{array}\right.
\end{split}
\]
where $0\leq \eta_{r}=o(r^p)$ as  $r\to 0^+$.
\end{enumerate}
\end{lemma}
\begin{proof}
The proof follows as in Lemma 9.7 in \cite{dTLi20}. For part \eqref{lem:propscheme-item-b} it is essential to use the fact that $J_p$ is a H\"older  continuous function, the basic properties of $\limsup$ and $\liminf$ and the consistency of $\MVPrh$ given in Theorem \ref{thm:MVF}.
\end{proof}

\subsection{Convergence} \label{sec:conv} We are now ready to prove the convergence stated in Theorem \ref{thm:dpp}. The idea of the proof originates from \cite{BS91}. The proof is almost the same as the proof of Theorem 2.5 ii) in \cite{dTLi20}.  We point out that it was necessary to adapt the proof in order to make it fit with the definition of viscosity solutions in the case $p\in (1,2)$. Below, we spell out some details. 

First we need another definition of viscosity solutions of the boundary value problem and two auxiliary results that are taken from \cite{dTLi20}.

\begin{definition}[Generalized viscosity solutions of the boundary value problem] \label{def:genvisc}
Let $f\in C(\overline{\Omega})$ and $g\in C(\partial \Omega)$. We say that a lower (resp. upper) semicontinuous function $u$ in $\overline{\Omega}$  is a \emph{generalized viscosity supersolution} (resp. \emph{subsolution}) of \eqref{eq:BVP}-\eqref{eq:BC} in $\overline{\Omega}$ if whenever $x_0 \in \overline{\Omega}$ and $\varphi \in
  C^2( B_R(x_0))$ for some $R>0$ are such that $|\nabla \varphi(x)|\neq 0$ for $x\in B_R(x_0)\setminus\{x_0\}$, 
$$ \varphi(x_0) = u(x_0) \quad \text{and} \quad \varphi(x) \leq
u(x) \ \text{(resp. $\varphi(x)\geq u(x)$)}\quad \text{for all} \quad
x \in  B_R(x_0)\cap\overline{\Omega},$$
then we have
\begin{equation*}
\begin{split}
\lim_{\rho\to 0}\sup_{B_{\rho(x_0)}\setminus\{x_0\}}\left(-\plap\varphi (x)-f(x_0)\right)&\geq 0 \quad \textup{if} \quad x_0\in \Omega\\
 \text{(resp. } \lim_{\rho\to 0}\inf_{B_{\rho(x_0)}\setminus\{x_0\}}\left( -\plap\varphi (x)-f(x_0)\right)& \leq 0\text{)}\\
\max\left\{\lim_{\rho\to 0}\sup_{B_{\rho(x_0)}\setminus\{x_0\}}\left(-\plap\varphi (x)-f(x_0)\right), u(x_0)-g(x_0)\right\}&\geq0 \quad \textup{if} \quad x_0\in \partial\Omega\\
\Big(\text{resp. } \min\left\{\lim_{\rho\to 0}\inf_{B_{\rho(x_0)}\setminus\{x_0\}}\left( -\plap\varphi (x)-f(x_0)\right), u(x_0)-g(x_0)\right\}&\leq0\Big)
\end{split}
\end{equation*}
\end{definition} 
\begin{remark}
 As in Remark \ref{rem:equivviscform}, we note that when either $p\geq2$ or $|\nabla \varphi(x_0)|\not=0$, the limits in the above definition can simply be replaced by $(-\plap \varphi (x_0)- f(x_0))$.
\end{remark}

The following uniqueness result is Theorem 9.5 in \cite{dTLi20}.

\begin{theorem}[Strong uniqueness property]\label{thm:SUP}
Let $\Omega$ be a bounded $C^2$ domain. If $u$ and $v$ are generalized viscosity subsolutions and supersolutions of \eqref{eq:BVP}-\eqref{eq:BC} respectively, then $u\leq v$.\end{theorem}

We also need that a generalized viscosity solution is a (usual) viscosity solution in the case of a bounded $C^2$ domain. The proposition below is Proposition 9.6 in \cite{dTLi20}.

\begin{proposition}\label{prop:equivnotions}
Let $\Omega$ be a bounded $C^2$ domain. Then $u$ is a viscosity subsolution (resp. supersolution) of \eqref{eq:BVP}-\eqref{eq:BC} if and only if $u$ is a generalized viscosity subsolution (resp. supersolution) of \eqref{eq:BVP}-\eqref{eq:BC}.
\end{proposition}

\begin{proof}[~Proof of Theorem \ref{thm:ConvNumSch}\eqref{thm:ConvNumSch-item2}]  
 Define 
\begin{equation*}
\overline{u}(x)=\limsup_{r\to 0,y\to x} u_h(y), \qquad \underline{u}(x)=\liminf_{r\to 0, y\to x} u_h(y), 
\end{equation*}
where $h\to 0$ as in the hypotheses of Theorem \ref{thm:ConvNumSch}. By definition $\underline{u}\leq\overline{u}$ in $\overline{\Omega}$. If we show that $\overline{u}$ (resp. $\underline{u}$) is a generalized viscosity subsolution (resp. supersolution) of \eqref{eq:BVP}, Theorem \ref{thm:SUP} would imply $\overline{u}\leq \underline{u}$. Thus, $u:=\overline{u}=\underline{u}$ is a generalized viscosity solution of \eqref{eq:BVP} and  $u_h\to u$ uniformly in $\overline{\Omega}$. Proposition \ref{prop:equivnotions} then would imply that $u$ is a viscosity solution of \eqref{eq:BVP}. 

We now sketch how to show that $\overline{u}$ is a generalized viscosity subsolution. First note that  $\overline{u}$ is an upper semicontinuous function by definition, and it is also bounded since $u_h$ is uniformly bounded by Proposition \ref{Prop:Stability}.
 Take $x_0\in \overline{\Omega}$ and $\varphi\in C^2(B_R(x_0))$ such that $\overline{u}(x_0)=\varphi(x_0)$, $\overline{u}(x)<\varphi(x_0)$ if $x\not=x_0$.
We separate the proof into different cases depending of the value of the gradient of $\varphi$ at $x_0$ and the range of $p$.
 
\emph{Case 1: $|\nabla \varphi(x_0)|\not=0$ or $p\geq2$}.  Then, for all $x\in \overline{\Omega}\cap B_R(x_0)\setminus\{x_0\}$, we have that
\begin{equation}\label{x0localmax}
 \overline{u}(x)-\varphi(x)<0=  \overline{u}(x_0)-\varphi(x_0).
 \end{equation}
 We claim that we can find a sequence $(r_n,y_n)\to(0,x_0)$ as $n\to \infty$, with $h_n\to 0$ as in the hypotheses of the theorem, such that
 \begin{equation}\label{eq:famseq}
 u_{h_n}(x)-\varphi(x) \leq u_{h_n}(y_n)-\varphi(y_n)+ e^{-1/r_n}   \quad \textup{for all} \quad x\in \overline{\Omega}\cap B_R(x_0).
 \end{equation}
 This can be argued for as in the proof of Theorem 2.5 ii) in \cite{dTLi20}.
 \begin{comment}
 To show this, we consider a sequence $(r_j,x_j)\to(0,x_0)$ as $j\to \infty$ such that $h_j=o(r_j^{3/2})$ and $\U_{h_j}(x_j)\to \overline{u}(x_0)$, which exists by definition of $\overline{u}$. By the definition of the supremum we can for each $j$, find $y_j$ such that
 \begin{equation}\label{eq:exp}
  \U_{h_j}(y_j)-\varphi(y_j)+ e^{-1/r_j}  \geq \sup_{\overline{B}_R(x_0)}\{\U_{h_j}-\varphi\}.
  \end{equation}
Now extract a subsequence $(r_n,x_n,y_n)\to (0, x_0, \hat{y})$ as $n\to \infty$ for some $\hat{y}\in \overline{\Omega}$. Then,
 \begin{equation*}
 \begin{split}
 0&= \overline{u}(x_0)-\varphi(x_0)\\
 &= \lim_{n\to\infty} \left\{\U_{h_n}(x_n)-\varphi(x_n)\right\}\\
 &\leq \limsup_{n\to\infty}\left\{\U_{h_n}(y_n)-\varphi(y_n)+ e^{-1/r_n}\right\}\\
 &\leq \limsup_{r\to0, y \to \hat{y}}\left\{U_{r}(y)-\varphi(y)+ e^{-1/r}\right\}\\
 &= \overline{u}(\hat{y})- \varphi(\hat{y}),
 \end{split}
 \end{equation*}
 where we in the third inequality have used \eqref{eq:exp}. This together with \eqref{x0localmax} implies that $\hat{y}=x_0$ and thus finishes proof of the claim. 
 \end{comment}
 Choose now $\xi_n:=u_{h_n}(y_n)-\varphi(y_n)$. We have from \eqref{eq:famseq} that,
 \[
 u_{h_n}(x)\leq \varphi(x) + \xi_n + e^{-1/r_n} \quad \textup{for all} \quad x\in \overline{\Omega}\cap B_R(x_0).
 \]
Using Lemma \ref{lem:propscheme}\eqref{lem:propscheme-item-a} we obtain
 \begin{equation*}
 \begin{split} 
 0&=S(r_n,h_n, y_n, u_{r_n}(y_n),u_{h_n})\\
 &=S(r_n,h_n, y_n, \varphi(y_n)+\xi_n,u_{h_n})\\
 &\geq S(r_n,h_n, y_n, \varphi(y_n)+\xi_n,\varphi + \xi_n + e^{-1/r_n} )\\
 &=S(r_n,h_n,y_n, \varphi(y_n)+\xi_n- e^{-1/r_n} ,\varphi + \xi_n ).
 \end{split}
 \end{equation*}
 Note that  $e^{-1/r}=o(r^p)$. By Lemma \ref{lem:propscheme}\eqref{lem:propscheme-item-b}, we have
  \begin{equation*}
 \begin{split} 
 0&\geq \liminf_{r_n\to0,\, y_n\to x_0,\, \xi_n\to0}S(r_n,h_n,y_n, \varphi(y_n)+\xi_n- e^{-1/r_n} ,\varphi + \xi_n )\\
 &\geq \liminf_{r\to0,\, y\to x_0,\, \xi \to0}S(r,h,y, \varphi(y)+\xi- e^{-1/r} ,\varphi + \xi )\\
 &\geq \left\{\begin{array}{cccl}
-\Delta_p\varphi(x_0)-f(x_0)& \text{ if } &x_0\in\Omega,\\
\min\{-\Delta_p\varphi(x_0)-f(x_0), \overline{u}(x_0)-g(x_0)\}& \text{ if } &x_0\in \partial{\Omega},
\end{array}\right.
 \end{split}
 \end{equation*}
 which shows that $\overline{u}$ is a viscosity subsolution and finishes the proof in this case.
 
\emph{Case 2: Let $p\in(1,2)$ and $|\nabla \varphi(x_0)|=0$ such that $\overline u$ is constant in some ball $B_\rho(x_0)$ for $\rho>0$ small enough.}   
Choose $\phi(x)=\overline u(x_0)+|x-x_0|^{\frac{p}{p-1}+1}$.  Then, we can argue as in Case 1 above that
\[
0 \geq \liminf_{r\to0,\, y\to x_0,\, \xi \to0}S(r,h,y, \phi(y)+\xi- e^{-1/r} ,\phi + \xi ),
\]
which implies
\[
0 \geq \liminf_{r\to0,\, y\to x_0}S(r,h,y, \phi(y) ,\phi ),
\]
by the H\"older continuity of $J_p$. Together with Lemma \ref{lem:xbeta} this shows that
 \[
-\Delta_p \overline u(x_0)=0\leq f(x_0).
\]
Hence,  $\overline u$ is a classical subsolution at $x_0$ and thus also a viscosity subsolution.

\emph{Case 3:  Let $|\nabla \varphi(x_0)|=0$ and assume that $\overline u$ is not constant in any ball $B_\rho(x_0)$}.  Then we may argue as in the proof of Proposition 2.4 in \cite{AR18} to prove that there is a sequence $y_k\to 0$ such that the function $\varphi_k(x)=\varphi(x+y_k)$ touches $\overline u$ from above at $x_k=x_0+y_k$ and $|\nabla \varphi_k(x_k)|\neq 0$ for all $k$. As in Case 1, this gives
\[
 \begin{split}
 0&\geq  \left\{\begin{array}{cccl}
\displaystyle-\Delta_p\varphi(x_k)-f(x_k)& \text{ if } &x_k\in\Omega,\\
\displaystyle\min\{-\Delta_p\varphi(x_k)-f(x_k), \overline{u}(x_k)-g(x_k)\}& \text{ if } &x_k\in \partial{\Omega},
\end{array}\right.
\end{split}
\]
for all $k$. Passing $k\to \infty$, we obtain
\[
 \begin{split}
 0&\geq  \limsup_{k\to\infty} \left\{\begin{array}{cccl}
\displaystyle(-\Delta_p\varphi(x_k)-f(x_k))& \text{ if } &x_k\in\Omega,\\
\displaystyle\min\{-\Delta_p\varphi(x_k)-f(x_k), \overline{u}(x_k)-g(x_k)\}& \text{ if } &x_k\in \partial{\Omega},
\end{array}\right.\\
&\geq \left\{\begin{array}{cccl}
\displaystyle \lim_{\rho\to 0}\inf_{B_{\rho}(x_0)\setminus\{x_0\}}\left( -\plap\varphi (x)-f(x)\right)& \text{ if } &x_0\in\Omega,\\
\displaystyle\min\{ \lim_{\rho\to 0}\inf_{B_{\rho}(x_0)\setminus\{x_0\}}\left( -\plap\varphi (x)-f(x)\right), \overline{u}(x)-g(x)\}& \text{ if } &x_0\in \partial{\Omega},
\end{array}\right.
\end{split}
\]

which is the desired inequality. This completes the proof.
 \end{proof}

\section{Solution of the nonlinear system}\label{sec:NonlinearSis}
	When we discretize the Dirichlet problem \eqref{eq:BVP}-\eqref{eq:BC}, we need to solve the nonlinear system \eqref{eq:NumDisc1}-\eqref{eq:NumDisc2}. In contrast to the situation in \cite{ObermanpLap}, our system is not based on the mean value formula for the $\infty$-Laplacian which is not differentiable. Instead, it is based on an implicit and differentiable mean value property. This system is therefore well suited for Newton-Raphson, which is one of the methods we have employed. We have also chosen to use an explicit method based on the convergence to a steady state of an evolution problem, for which we can guarantee the convergence. The Newon-Raphson method is fast (as explained by Oberman in \cite{ObermanpLap}) since the number of iterations required to solve the system is independent of its size. This is not the case for our explicit method that is conditioned by the CFL-type condition \eqref{eq:CFL} in Section \ref{subsec:explicit}. See Table \ref{table:nonlinsys} for a more detailed comparison between the efficiency in terms of speed of the two methods. We describe the two methods in detail below.

\subsection{Newton-Raphson} The method we have used is the standard one. Let $F:\R^k \to \R^k$ for some $k\geq1$. In order to solve the system 
$$
F (z) = 0, 
$$
we use the iteration 
$$
z_{n+1}=z_n-( J_F(z_n))^{-1} F(z_n).
$$
where $J_F$ denotes the Jacobian matrix of the function $F$. In our particular case we have that $k=\# \{\G_h \cap \Omega_r$\}.

Let us illustrate the form of $F$ and $J_F$ in the one dimensional case. Let  $\gamma=\min\{\beta\in \Z \ : \ x_\beta\in \Omega_r\}$, and $z_i=U_{\gamma + i-1}$. Consider
\[
F(z_1,\ldots,z_k)=\left(
\begin{matrix}
F_1(z_1,\ldots,z_k)\\
F_2(z_1,\ldots,z_k)\\
\vdots\\
F_k(z_1,\ldots,z_k)
\end{matrix}
\right)
\] 
where $F_i:\R^k\to\R$ for $i=1,\ldots,k$ are given by
\[
F_i(z_1,\ldots,z_k)=
\begin{cases}
z_i-G_{\gamma+i-1}\quad &\textup{if} \quad x_{\gamma+i-1}\in \partial \Omega_r,\\
\frac{h^d}{D_{d,p}\, \omega_d\, r^{p+d}} \displaystyle\sum_{x_\alpha\in B_r} J_p(z_{i+\alpha}-z_i) - f_{\gamma+i-1} \quad &\textup{if} \quad x_{\gamma+i-1}\in \Omega.
\end{cases}
\]
Let $(J_F(z))_{i,j}=(J_F(z_1,\ldots,z_k))_{i,j}$ denote the component of the Jacobian matrix of $F$ corresponding to the $i$-th and $j$-th column. If $i$ is such that $x_{\gamma+i-1}\in\partial \Omega_r$ then
\[
(J_F(z))_{i,j}=\begin{cases}
1 \quad \textup{if} \quad j=i\\
0 \quad \textup{if} \quad j\not=i
\end{cases}
\] 
while if $x_{\gamma+i-1}\in \Omega$ then
\[
(J_F(z))_{i,j}=\frac{(p-1)h^d}{D_{d,p}\, \omega_d\, r^{p+d}}\times\begin{cases}
|z_{i+j}-z_i|^{p-2} \quad &\textup{if} \quad j\not =i \quad \textup{and} \quad  x_j\in B_r,\\
\displaystyle-\sum_{x_\alpha \in B_r}|z_{i+\alpha}-z_i|^{p-2}  \quad &\textup{if} \quad j=i,\\
0 &\textup{otherwise}.
\end{cases}
\]

\subsection{Explicit method}\label{subsec:explicit}
We consider $\{U^m\}_{m\in \N}$ to be the sequence of solutions $U^m:\Omega^h_r\to \R$ of 
\begin{equation}
\label{eq:expsch}
U^{m+1}_\beta= U^m_\beta+\tau_m \MVPrh U^m_\beta+\tau_m  f_\beta, \quad x_\beta \in \Omega^h
\end{equation}
where $U^0$ is some initial data, $U^m = G$ on $\partial\Omega_{r}^h$ and $\{\tau_m\}_{m\in \N}>0$ are certain discretization parameters.  The idea here is that, as $m\to\infty$, $U^m$ converges to the solution $U$ of \eqref{eq:NumDisc1}-\eqref{eq:NumDisc2}. 
This convergence holds given a nonlinear counterpart to the CFL-stability condition. Actually, we also need to slightly modify \eqref{eq:expsch} to ensure convergence; in words of Oberman in \cite{Ob2}, we need to ensure that our operator is \em proper\em.  

More precisely, given $\veps>0$, let $\{(U_\veps)^m\}_{m=1}^\infty$ be the solution of
\begin{equation}
\label{eq:expscheps}
(U_\veps)^{m+1}_\beta= (U_\veps)^m_\beta+\tau_m \MVPrh (U_\veps)^m_\beta- \tau_m\veps (U_\veps)^m_\beta+\tau  f_\beta, \quad x_\beta \in \Omega^h
\end{equation}
subject to the same initial and boundary conditions as in \eqref{eq:expsch}.
Let $U_\veps$ be the solution of  
\begin{empheq}[left=\empheqlbrace]{align}
\label{eq:NumDisceps1}
-\MVPrh (U_\veps)_\beta+ \veps (U_\veps)_\beta=  f_\beta, &\quad  x_\beta\in \Omega^h,\\
\label{eq:NumDisceps2}\ (U_\veps)_\beta=G_\beta, &\quad  x_\beta\in \partial \Omega_r^h.
\end{empheq}
It is standard to check, using the techniques of Section \ref{sec:existanduniquenumsol} that $U_\veps$ exists, is unique and uniformly bounded in $r,h$ and $\veps$. 
We have the following result.
\begin{lemma}
Let $p\geq2$ and $\{U_\veps^m\}_{m=1}^\infty$ be the solution of \eqref{eq:expscheps} with any bounded initial condition $U_\veps^0$. Let also $U$ be the solution of \eqref{eq:NumDisc1}-\eqref{eq:NumDisc2}. Assume that 
\begin{equation}\label{eq:CFL}\tag{CFL}
0<\tau_m\leq \min\left\{1, \frac{r^p}{(p-1)2^{p-2}L_m^{p-2}}\frac{D_{d,p}}{(1+\sqrt{d})^d}(1-\veps)\right\} \quad \textup{with} \quad L_m = \max( \|U^m_\veps\|_{\ell^\infty}, \|   U_\veps \|_{\ell^\infty}  ). 
\end{equation}
Then
\[
\max_{x_\alpha\in \Omega}\left|(U_\veps)^m_\alpha-U_\alpha\right|=2L_0 (1-\tau\veps)^m + o_\veps(1),
\]
where $\tau=\inf_{m\in \N} \{\tau_m\}$.
\end{lemma}
\begin{proof} 
Since $U_\veps$ is uniformly bounded in a discrete finite set, there exists a convergent subsequence $U_{\veps_j}$ converging to some $V$ pointwise. It is also standard to show that $V$ is indeed a solution of \eqref{eq:NumDisc1}-\eqref{eq:NumDisc2}. By uniqueness, $V=U$ and the full sequence $U_{\veps}$ converges, i.e.,
\[
\|U_{\veps}-U\|_\infty=o_\veps(1).
\]
On the other hand, by subtracting  the equations for $U_\veps$ and $(U_\veps)^m$ we get
\begin{equation*}
\begin{split}
(U_\veps&)^{m+1}_\beta-(U_\veps)_\beta \\
&= ((U_\veps)^{m}_\beta-(U_\veps)_\beta)(1-\tau_m\veps)  + \tau_m K\sum_{y_\alpha\in B_r}\left( J_p((U_\veps)^m_{\beta+\alpha}-(U_\veps)^m_{\beta})-J_p((U_\veps)_{\beta+\alpha}-(U_\veps)_{\beta}) \right) \\
&= ((U_\veps)^{m}_\beta-(U_\veps)_\beta)(1-\tau_m\veps)  + \tau_m K\sum_{y_\alpha\in B_r} J'_p(\xi_{\alpha,\beta})\left( ((U_\veps)^m_{\beta+\alpha}-(U_\veps)^m_{\beta}) -((U_\veps)_{\beta+\alpha}-(U_\veps)_{\beta}) \right) \\
&= ((U_\veps)^{m}_\beta-(U_\veps)_\beta)\left(1-\tau_m\veps- \tau_m K\sum_{y_\alpha\in B_r}J'_p(\xi_{\alpha,\beta})\right) +  \tau_m K\sum_{y_\alpha\in B_r} J'_p(\xi_{\alpha,\beta})\left( (U_\veps)^m_{\beta+\alpha}-(U_\veps)_{\beta+\alpha} \right)
\end{split}
\end{equation*}
where $K:=\frac{h^d}{|B_r|D_{d,p}r^p}$ and $\xi_{\alpha,\beta}$ lies between $(U_\veps)^m_{\beta+\alpha}-(U_\veps)^m_\beta$ and $(U_\veps)_{\beta+\alpha}-(U_\veps)_\beta$, so that $|\xi_{\alpha,\beta}|\leq 2L_m$ and $|J_p'(\xi_{\alpha,\beta})|=(p-1)|\xi_{\alpha,\beta}|^{p-2}\leq (p-1)2^{p-2}L^{p-2}_m$ since $p\geq2$.
Therefore, when $r$ is small enough
\begin{equation*}
\begin{split}
\tau_m K\sum_{y_\alpha\in B_r}J'_p(\xi_{\alpha,\beta})\leq(1-\veps)\frac{1}{|B_r|}\frac{1}{(1+\sqrt{d})^d}\sum_{y_\alpha\in B_r} h^d\leq (1-\veps) \frac{| B_{r+\sqrt{d}h}|}{|B_{r+\sqrt{d}r}|}\leq  (1-\veps),
\end{split}
\end{equation*}
where we used \eqref{eq:CFL} and that $h\leq r$ (since $h=o(r)$ with $r$ small enough). In this way,
\[
1-\tau_m\veps- \tau_m K\sum_{y_\alpha\in B_r}J'_p(\xi_{\alpha,\beta})\geq \veps(1-\tau_m)\geq0.
\]
Clearly $J_p'\geq0$, and then
\begin{equation*}
\begin{split}
\|&(U_\veps)^{m+1}-U_\veps\|_{\ell^\infty} \\
&\leq \|(U_\veps)^{m}-U_\veps\|_{\ell^\infty}\left(1-\tau_m\veps- \tau_m K\sum_{y_\alpha\in B_r}J'_p(\xi_{\alpha,\beta})\right) +  \tau_m K\sum_{y_\alpha\in B_r} J'_p(\xi_{\alpha,\beta}) \|(U_\veps)^{m}-U_\veps\|_{\ell^\infty}\\
&\leq\|(U_\veps)^{m}-U_\veps\|_{\ell^\infty}\left(1-\tau_m\veps\right)\\
&\leq\|(U_\veps)^{0}-U_\veps\|_{\ell^\infty}\left(1-\tau_m\veps\right)^{m+1}\\
&\leq  2L_0\left(1-\tau \veps\right)^{m+1}.
\end{split}
\end{equation*}
The results follows using the triangle inequality:
\[
\left\|(U_\veps)^m-U\right\|_{\ell^\infty}\leq \|(U_\veps)^m-U_\veps\|_{\ell^\infty} +\|{U}_{\veps}-U\|_{\ell^\infty}\leq  2L_0 (1-\tau\veps)^m + o_\veps(1).
\]
\end{proof}

\begin{remark}
The fact that $U_\veps$ is uniformly  bounded together with the bound $\|U^{m}_\veps-U_\veps\|_\infty \leq 2L_0$ ensures that $L_m$  is uniformly bounded from above so that  $\{\tau_m\}_{m\in N}$ can be taken uniformly bounded from below.  
\end{remark}
 
 In the case $1<p<2$, we used a regularization of the singularity in $\MVPrh$ in order to make it a Lipschitz map. This could be done for example by modifying the nonlinearity with an extra approximation parameter $\delta>0$ and replacing $J_p$ by $J_p^\delta$ given by
\[
J_p^\delta(t)=\left\{\begin{split}
&J_p(t+\delta) -J_p(\delta) \quad \ \ \ \textup{if} \quad t\geq0,\\
&J_p(t-\delta) -J_p(-\delta) \quad \textup{if} \quad t<0.
\end{split}\right.
\]
The drawback of this type of regularization is that the condition \eqref{eq:CFL} becomes more and more restrictive as $\delta\to0$. This regularization is typically used when dealing with explicit schemes for fast diffusion equations (see for example \cite{dTEnJa18, dTEnJa19})

\subsection{Comparison between the solvers} 
 We now present a comparison of the above methods regarding the number of iterations and computational time\footnote{ Naturally, this depends on the code and the computational power of the computer used, but we have chosen to include it for the sake of completeness.}. 

 We have solved the system \eqref{eq:NumDisc1}-\eqref{eq:NumDisc2} for $p=3$, in dimension $d=1$ with $\Omega=(-1,1)$, $f\equiv 1$ and $g\equiv0$. As starting value for the iteration we have chosen  $u_0(x)=(1-|x|)_+$. Finally, for the explicit solver we have chosen $\tau$ to satisfy \eqref{eq:CFL}.
We have stopped the solver when difference between two consecutive iterations is less that $10^{-16}$.

In Table \ref{table:nonlinsys} we present the results for different values of $r$ and its corresponding $h$ satisfying \eqref{eq:ass} (in this case $h=\frac{r^{3/2+0.1}}{4}$).

\begin{table}[h!]
\centering
\begin{tabular}{l|l|l|l|l|l|l|l}
 &   $r$   &$h$ & $k$ & It-E & T-E& It-NR & T-NR \\ \hline
& 0.2 & 0.019037 & 127 & 4272 & 0.59 & 8 & 0.03 \\
& 0.1& 0.006279 & 351 & 17475 & 9.74 & 8 & 0.1 \\
& 0.05& 0.002071 & 1014 & 63164  & 166.87 & 9 & 0.84 \\
& 0.025& 0.000683 & 3000 & 250901 & 3076.04&  9 & 11.43 \\
& 0.0125& 0.000025 & 8984 & $\sim 10^7 $ &$\sim 10^5$ & 9 & 381.28                                       
\end{tabular}
\vspace{4mm}

\caption{A comparison of the efficiency of our methods used to solve the nonlinear system for $p=3$. Here $k$ denotes the size of the system, It-E and It-NR are the number of iterations needed by the explicit and the Newton-Raphson solver respectively, and T-E and T-NR are the times (in seconds) spent to solve the system by the explicit and the Newton-Raphson solver respectively.}
\label{table:nonlinsys}
\end{table}

As the table shows, the Newton-Raphson solver is fast in the sense that the number of iterations does not depend in the size of the system. This is a big advantage compared to the explicit solver, for which smaller values of $r$ enforces smaller choices of $\tau$ which increase the number of iterations required substantially.

\section{Numerical experiments}\label{sec:exp}

To perform numerical experiments we need two ingredients. 
\begin{enumerate}
\item\label{numtest-item1} The explicit value of the constant $D_{d,p}$.
\item\label{numtest-item2} Explicit solutions of \eqref{eq:BVP}-\eqref{eq:BC} to test with.
\end{enumerate}
 It is standard to check
that, in dimension $d=1$, we have
\[
D_{1,p}=\frac{1}{2(1+p)}.
\]
In dimension $d=2$ the constant is not so explicit in general, but we have the following result allowing us to compute it for integer numbers, which partially answers \eqref{numtest-item1}.
\begin{lemma}
Let $p\in[2,\infty)$ and $d=2$. 
\begin{enumerate}[\rm (a)]
\item (Even) If $p=2n$ for some $n\in \N$ then 
\[
D_{2,p}= \frac{1}{2+p}\left(\prod_{i=1}^n \frac{2i-1}{2i}\right).
\]
\item (Odd) If $p=2n+1$ for some $n\in \N$ then 
\[
D_{2,p}= \frac{2}{\pi(2+p)}\left(\prod_{i=1}^n \frac{2i}{2i+1}\right).
\]
\end{enumerate} 
\end{lemma}

\begin{proof}
In dimension $d=2$ we have
\[
D_{2,p}=\frac{1}{2+p} \int_{\partial B_1} |y_1|^p \dd \sigma(y)= \frac{1}{\pi (2+p)}\int_{-\frac{\pi}{2}}^{\frac{\pi}{2}} (\cos(\theta))^p \dd \theta.
\]
Now we note that for $p\geq2$ a simple integration by parts yields 
\[
\begin{split}
I_p:=&\int_{-\frac{\pi}{2}}^{\frac{\pi}{2}} (\cos(\theta))^p \dd \theta = \int_{-\frac{\pi}{2}}^{\frac{\pi}{2}} (\cos(\theta))^{p-1} \cos(\theta) \dd \theta= (p-1)\int_{-\frac{\pi}{2}}^{\frac{\pi}{2}} (\cos(\theta))^{p-2} (\sin(\theta))^2 \dd \theta\\
=& (p-1) I_{p-2} - (p-1) I_p.
\end{split} 
\]
So we have the recurrence relation $I_p=\frac{p-1}{p}I_{p-2}$. We only need to compute 
$$I_0=\int_{-\frac{\pi}{2}}^{\frac{\pi}{2}} \dd \theta=\pi\quad \text{and}\quad  I_1=\int_{-\frac{\pi}{2}}^{\frac{\pi}{2}} \cos(\theta) \dd \theta=2.$$
This finishes the proof. 
\end{proof}

 As mentioned in the introduction, homogeneous problems can successfully be treated by means of the so-called normalized $p$-Laplacian, for which numerical schemes are well understood (see \cite{Obe05, ObermanpLap}). Therefore, we will focus on nonhomogeneous problems ($f\not=0$). We compare our numerically obtained solution with the explicit solution
\[
u(x)= (1-|x|^{\frac{p}{p-1}}) \frac{p-1}{p}\frac{1}{d^{\frac{1}{p-1}}} .
\]
Note that $u$ is a solution of 
\begin{equation}\label{eq:explicitsol}
\begin{cases}
-\lap_p u(x) =1, & x\in B_1,\\
u(x)=0, & x\in \partial B_1.
\end{cases}
\end{equation}

\subsection{Error analysis in dimension $d=1$}\label{sec:Numdim1} Here we present the results of a numerical experiment using our numerical scheme to solve problem \eqref{eq:explicitsol} in dimension $d=1$ using MATLAB. 

To solve the nonlinear system present in \eqref{eq:NumDisc1}-\eqref{eq:NumDisc2} we use the explicit solver given by \eqref{eq:expscheps}. The parameter $\tau_m$ has been chosen to satisfy the \eqref{eq:CFL}, while $\veps$ is chosen small enough to not interfere with the error in $h$ and $r$. We have also taken $G(x)=0$ for all $x\in\partial \Omega_r$ as extended boundary condition.

We have  stopped the explicit solver when it has reached a numerical steady state, i.e., 
\[
\max_{x_\alpha\in \Omega}|(U_\veps)^{m+1}_\alpha-(U_\veps)^{m}_\alpha|<10^{-16}.
\]
In this case we have chosen to take $h=r^2/4$ which clearly satisfy the condition $h=o(r^{3/2})$. The results obtained are presented in Figure \ref{fig:Dim1error} and Table \ref{table:Dim1error} which contain the simulations for $p=3$, $p=4$ and $p=10$.

\begin{figure}[h!]
\includegraphics[width=\textwidth]{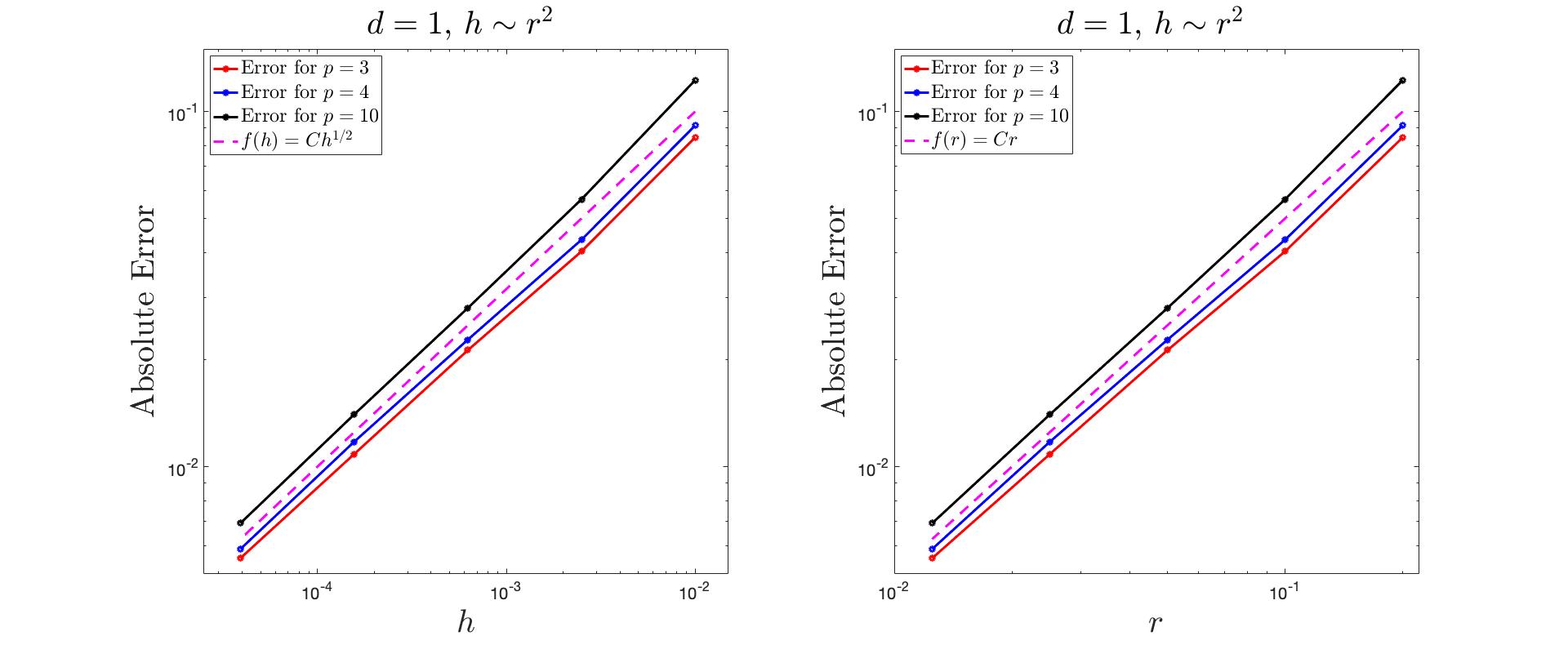}
\caption{$\ell^\infty$-absolute error $\|(U_\veps)_h-u\|_{\ell^\infty}$ and an approximated convergence rate in $r$ and $h$ (in rose) in dimension $d=1$ for problem \eqref{eq:explicitsol}}
\label{fig:Dim1error}
\end{figure}

It can be clearly seen that the error seems to behave linearly with $r$. This can be seen more clearly in Table \ref{table:Dim1error}, where we present the details of the results in Figure \ref{fig:Dim1error}.

\begin{table}[h!]
\centering
\begin{tabular}{l|l|ll|ll|ll|l}
&& \quad  $p=3  $                  &  &     \quad $p=4$                  &  &      \quad  $p=10$               &                        \\ \cline{3-8} 
 \quad $r$ &\   $h=\frac{r^2}{4}$& \multicolumn{1}{l|}{error} & $\gamma$  & \multicolumn{1}{l|}{error} & $\gamma$  & \multicolumn{1}{l|}{error} & $\gamma$   \\ \hline
$2.00$e-$1$&$1.000$e-$2$& \multicolumn{1}{l|}{$8.46$e-$2$} &  & \multicolumn{1}{l|}{$ 9.13$e-$2 $} &  & \multicolumn{1}{l|}{$ 1.23$e-$1 $} &                \\
$1.00$e-$1$&$2.500$e-$3$& \multicolumn{1}{l|}{$4.03 $e-$2 $} &1.07  & \multicolumn{1}{l|}{$4.35 $e-$2 $} &1.07  & \multicolumn{1}{l|}{$ 5.66 $e-$2 $} &1.12                      \\
$5.00$e-$2$&$6.250$e-$4$& \multicolumn{1}{l|}{$2.13 $e-$2 $} &0.92 & \multicolumn{1}{l|}{$2.27 $e-$2 $} &0.94 & \multicolumn{1}{l|}{$ 2.80 $e-$2 $} &1.02                 \\
$2.50$e-$2$&$1.563$e-$4$& \multicolumn{1}{l|}{$1.08$e-$2 $} &0.97  & \multicolumn{1}{l|}{$1.17 $e-$2 $} &0.96  & \multicolumn{1}{l|}{$2.40$e-$2 $}&1.00                 \\
1.25e-$2$&$3.906$e-$5$& \multicolumn{1}{l|}{$5.52$e-$3 $} &0.97  & \multicolumn{1}{l|}{$5.86 $e-$3 $} &1.00  & \multicolumn{1}{l|}{$6.93$e-$3 $} & 1.02                           
\end{tabular}
\vspace{4mm}
\caption{$\ell^\infty$-absolute error $\|(U_\veps)_h-u\|_{\ell^\infty}$ and observed convergence rate $\gamma$ in $r$ in dimension $d=1$ for problem \eqref{eq:explicitsol}.}
\label{table:Dim1error}
\end{table}
The observed convergence rate $\gamma$ have been computed in to be such that
\[
\textup{error}_j= k (r_j)^\gamma, \quad j=0,1,2,3,4.
\]
where $r_j=0.2/2^j$. In this way,
\[
\gamma= \log_2\left(\frac{\textup{error}_{j-1}}{\textup{error}_{j}}\right).
\]

\subsection{Error analysis in dimension $d=2$}\label{sec:Numdim2} We now perform numerical experiments in dimension $d=2$. We have   almost the same setup as in Section \ref{sec:Numdim1}, except that we now take $h=r^{\frac{3}{2}+0.1}$ which clearly satisfy the condition $h=o(r^{3/2})$.

\begin{figure}[h!]
\includegraphics[width=\textwidth]{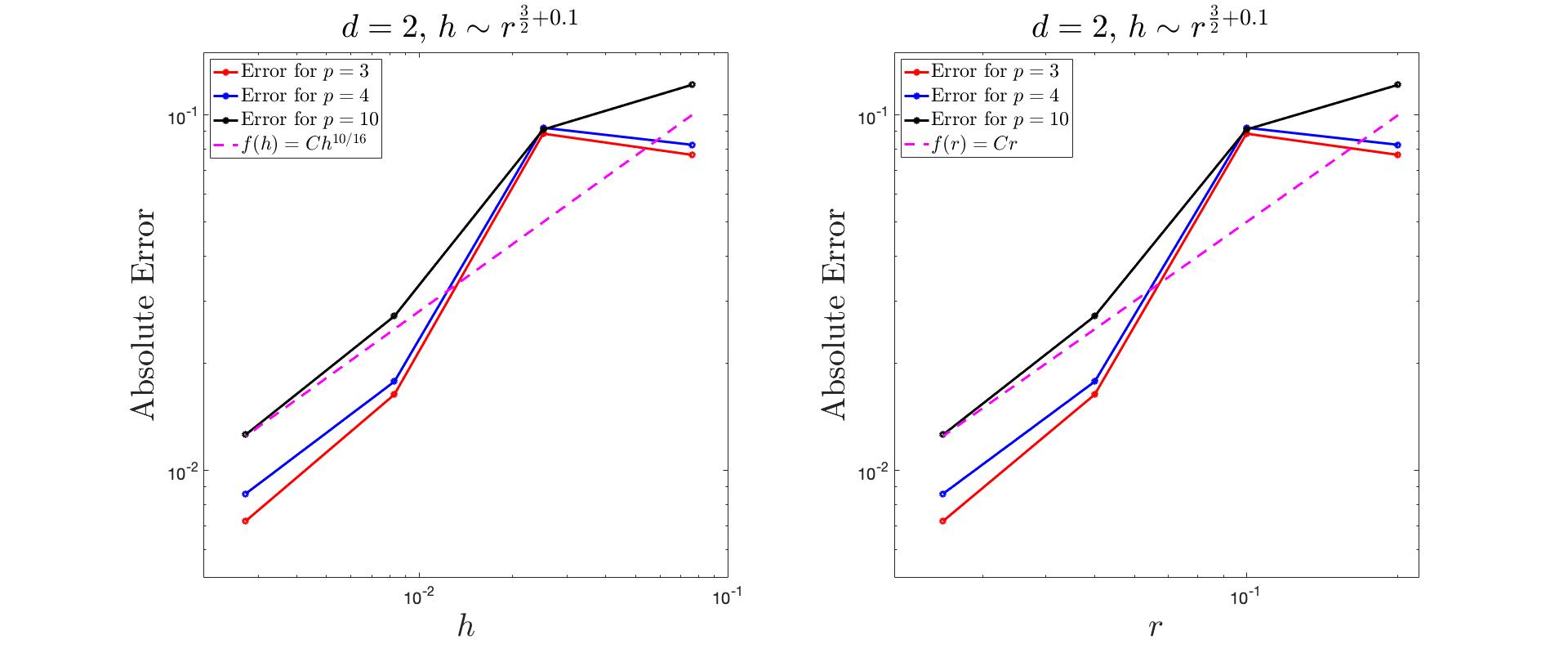}
\caption{$\ell^\infty$-absolute error $\|(U_\veps)_h-u\|_{\ell^\infty}$ and an approximated convergence rate in $r$ and $h$ (in rose) in dimension $d=2$ for problem \eqref{eq:explicitsol}}
\label{fig:Dim2error}
\end{figure}

Again, as in the computation in dimension $d=1$, the error observed in Figure \ref{fig:Dim2error} seems to decay at least linearly with $r$,  despite the fact that we have taken the parameter $h$ to  decay slower than before. It seems as if as long as $h=o(r^{3/2})$, the choice of $h$ does not interfere with the order of convergence in $r$.

\begin{table}[h!]
\centering
\begin{tabular}{l|l|ll|ll|ll|l}
 &   & \quad  $p=3  $                  &  &     \quad $p=4$                  &  &      \quad  $p=10$               &                        \\ \cline{3-8} 
\quad $r$ &\ $h=r^{\frac{3}{2}+0.1}$& \multicolumn{1}{l|}{error} & $\gamma$  & \multicolumn{1}{l|}{error} & $\gamma$  & \multicolumn{1}{l|}{error} & $\gamma$   \\ \hline
$2.00$e-$1$&$7.615$e-$2$& \multicolumn{1}{l|}{$7.73$e-$2$} &  & \multicolumn{1}{l|}{$ 8.25$e-$2 $} &  & \multicolumn{1}{l|}{$ 1.22$e-$1 $} &                \\
$1.00$e-$1$&$2.512$e-$2$& \multicolumn{1}{l|}{$8.87 $e-$2 $} &-0.20  & \multicolumn{1}{l|}{$9.21 $e-$2 $} &-0.15  & \multicolumn{1}{l|}{$ 9.11 $e-$2 $} &0.41                      \\
$5.00$e-$2$&$8.286$e-$3$& \multicolumn{1}{l|}{$1.64 $e-$2 $} &2.44 & \multicolumn{1}{l|}{$1.78 $e-$2 $} &2.37 & \multicolumn{1}{l|}{$ 2.72 $e-$2 $} &1.74                 \\
$2.50$e-$2$&$2.733$e-$3$& \multicolumn{1}{l|}{$7.21$e-$3 $} &1.18  & \multicolumn{1}{l|}{$8.61 $e-$3 $} &1.05  & \multicolumn{1}{l|}{$1.26$e-$2 $}&1.11                                         
\end{tabular}
\vspace{4mm}
\caption{$l^\infty$-absolute error $\|(U_\veps)_h-u\|_{\ell^\infty}$ and observed convergence rate $\gamma$ in $r$ in dimension $d=2$ for problem \eqref{eq:explicitsol}.}
\label{table:Dim2error}
\end{table}

In Table \ref{table:Dim2error} we observe some instabilities in the order of convergence in the simulations for big choices of $r$ and $h$. However, if we compute the order of convergence between the simulation with $r=2.00$e-$1$ and $r=2.50$e-$2$ the observed rate is
\[
\begin{split}
\log_8\left(\frac{7.73\textup{e-}2}{7.21\textup{e-}3}\right)&= 1.14>1 \quad \textup{if} \quad p=3.
\\
\log_8\left(\frac{8.25\textup{e-}2}{8.61\textup{e-}3}\right)&= 1.09>1 \quad \textup{if} \quad p=4\\
\log_8\left(\frac{1.22\textup{e-}1}{1.26\textup{e-}2}\right)&= 1.09>1 \quad \textup{if} \quad p=10
\end{split}
\]
which is actually slightly better than linear in all the cases. 
 
\subsection{Improvement of the error with an adapted boundary condition} During the simulations presented in Section \ref{sec:Numdim1} and Section \ref{sec:Numdim2}, we observed that the extension of $G\equiv0$ produced a certain instability in the solution close to the boundary. Due to this fact, the maximal error  is attained near the boundary.

\begin{figure}[h!]
\includegraphics[width=\textwidth]{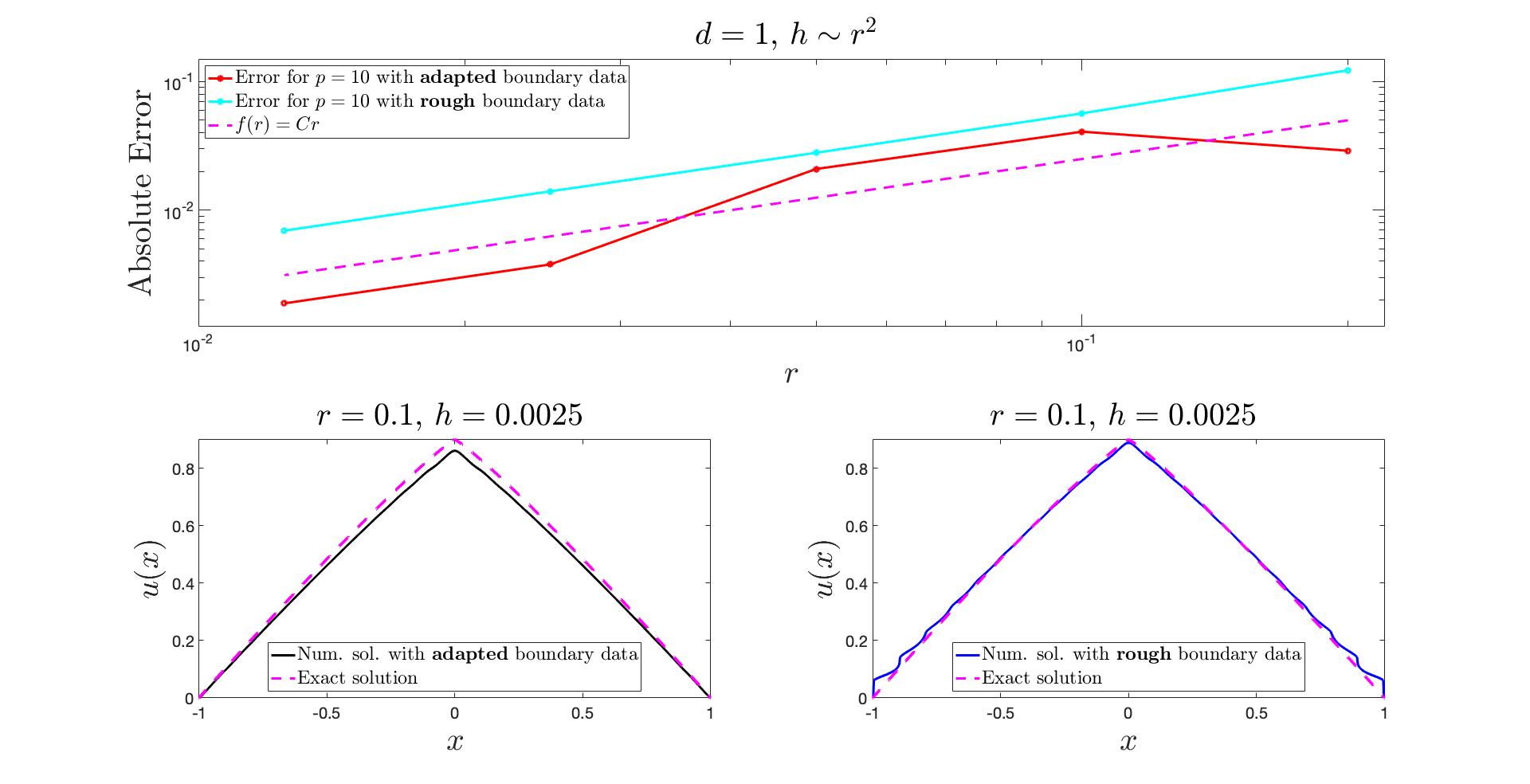}
\caption{Dimension $d=1$ and $p=10$. \emph{Top figure:} Error analysis for the adapted and rough boundary extension ($G\equiv0$) and the adapted boundary extension ($G$ given by \eqref{eq:Gadapted}). \emph{Bottom figures:} Representation of the numerical and real solutions for the two boundary extensions.
}
\label{fig:Dim1errorAdaptedVSNonadapted}
\end{figure}

In order to  avoid this  phenomenon, we have adapted the boundary condition to make the transition between the interior and the boundary smoother.  We have taken
\begin{equation}\label{eq:Gadapted}
G(x)=\frac{p-1}{p} (1-|x|^{\frac{p}{p-1}}) \quad \textup{for} \quad x\in \partial \Omega_r.
\end{equation}

In the results presented in Figure \ref{fig:Dim1errorAdaptedVSNonadapted}, we clearly see that the maximum error of the solution with an adapted condition comes from the middle point, which is the point where solution is the least regular, while without adaption, the error comes from the instabilities created near the boundary. 

Thus,  the correction seems to give a smoother transition between the interior and the extended condition. It also seems to improve the error estimate (but not the order of convergence).

\subsection{Solution of a fully nonhomogeneous problem.}\label{sec:fullynon} Finally, we present some numerical simulations of a problem with nonhomogeneous right hand side  and nonhomogeneous boundary conditions. We present the numerical solutions corresponding to problem \eqref{eq:BVP}-\eqref{eq:BC} in dimension $d=2$, posed in $\Omega=B_1(0)$ with $f\equiv $ constant in $\Omega$ and $g(x,y)=\frac{1}{2}+xy$ on $\partial \Omega$. 

The boundary condition has been extended to $\partial \Omega_r$ by $G(x,y)=\frac{1}{2}+xy$ and we have chosen the numerical parameters $r=0.2$ and $h=r^2=0.04$. 

 In Figure \ref{fig:Dim2fullyNonHomo}, we present a level set representation of the solutions for $p=1.1$, $p=1.5$, $p=2$, $p=4$ and $p=20$ (using the regularization described at the end of Section \ref{sec:NonlinearSis} when $p<2$). Here, $h=r^{2}$ has been used, also when $p<2$ (see Remark \ref{rem:conj}).

\begin{figure}[h!]
\includegraphics[width=\textwidth]{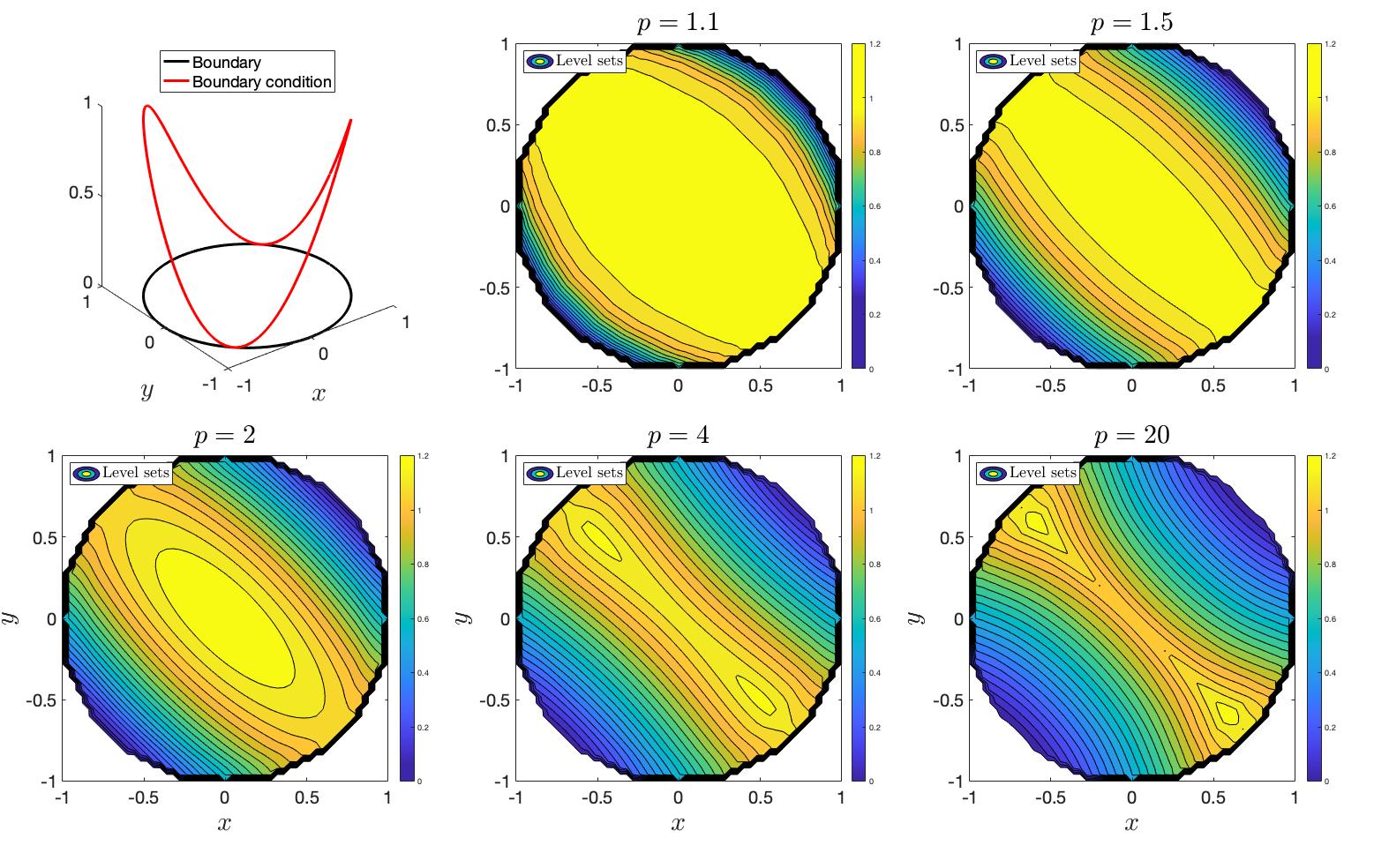}
\caption{Dimension $d=2$. Numerical solution of the fully nonhomogeneous problem \eqref{eq:BVP}-\eqref{eq:BC} with $g(x,y)=\frac{1}{2}+xy$ and $f(x,y)\equiv$ constant.}
\label{fig:Dim2fullyNonHomo}
\end{figure}

\section*{Acknowledgements}
F. del Teso was partially supported by PGC2018-094522-B-I00 from the MICINN of the Spanish Government. E. Lindgren is supported by the Swedish Research Council, grant no. 2017-03736.

\section{Appendix} 
\begin{lemma}{\label{lem:pineq1}}
Let $p\geq 2$. Then 
$$
\Big||a+b|^{p-2}(a+b)-|a|^{p-2}a\Big|\leq C\max( |a|,|a+b|)^{p-2}|b|, 
$$
where $C=C(p)$.
\end{lemma}
\begin{proof}
It follows from the fact that 
$$
|a+b|^{p-2}(a+b)-|a|^{p-2}a=(p-1)\int_0^b |a+s|^{p-2} \dd s.\qedhere
$$
\end{proof}

The following inequality is Lemma 3.4 in \cite{KKL}.
\begin{lemma}{\label{lem:pineq3}} Let $p\in (1,2)$. Then 
$$
\Big||a+b|^{p-2}(a+b)-|a|^{p-2}a\Big|\leq C\left(|a|+|b|\right)^{p-2}|b|.
$$
Here $C=C(p)$.
\end{lemma}

We also need the following lemma in the proof of convergence.

\begin{lemma}\label{lem:xbeta} Assume $p\in (1,2)$, \eqref{eq:ass} and let $\phi(x)=|x|^\beta$ with $\beta>p/(p-1)$. Then 
\[
\lim_{r\to 0,x\to 0}
\MVPrh\phi(x)=0. 
\]
\end{lemma}
\begin{proof} 
If $x=0$, we have $|\phi(x+y_\alpha)-\phi(x)|=|y|^\beta= o(|y_\alpha|^{\frac{p}{p-1}}) $. Then
\[
|\MVPrh\phi(x)| \leq J_p(o(|r|^{\frac{p}{p-1}}))  \frac{1}{|B_r| r^{p}} \sum_{y_\alpha\in B_r} h^d  = o_r(1).
\]
Assume now that $x\not=0$ so that $\nabla \phi(x)\not=0$. We can use the symmetry of $J_p(y_\alpha\cdot \nabla \phi(x))$ for $y_\alpha\in B_r$ and Lemma \ref{lem:pineq3} to conclude that
\begin{equation*}
\begin{split}
\left| \sum_{y_\alpha\in B_r}  J_p(\phi(x+y_\alpha)-\phi(x)) \dd y\right| &= \left|\sum_{y_\alpha\in B_r} J_p( y_\alpha\cdot \nabla\phi (x) + \frac{1}{2}y_\alpha^T D^2 \phi(\xi_\alpha)y_\alpha ) \right| \\
\leq& \sum_{y_\alpha\in B_r} \Big||y_\alpha\cdot \nabla\phi (x)| + |y_\alpha|^2 \sup_{\xi\in B_{r}(x)}| D^2 \phi(\xi)| \Big|^{p-2} |y_\alpha|^2\sup_{\xi\in B_{r}(x)}| D^2 \phi(\xi)|  \dd y.
\end{split}
\end{equation*}
We may assume that $x$ lies in the $e_1$-direction and write $\nabla \phi (x) =\beta|x|^{\beta-1} e_1:=c e_1$  for some $c>0$.  We now claim that\footnote{ Here $a \lesssim b$ stands for $a\leq Cb$ where $C$ is a constant that may depend on $p$ and $d$ but not on $r,h$ or $x$.}
\begin{equation}\label{eq:oufff}
\begin{split}
\MVPrh\phi(x) 
&\lesssim   \frac{c^{p-2} h^d}{|B_r|} \sum_{y_\alpha\in B_r}\Big||\hat y_\alpha\cdot e_1| + c^{-1}r\sup_{\xi\in B_{r}(x)}| D^2 \phi(\xi)| \Big|^{p-2} \sup_{\xi\in B_{r}(x)}| D^2 \phi(\xi)|\\
&\leq \frac{c^{p-2} h^d}{|B_r|}  \sum_{y_\alpha\in B_r}\Big||\hat y_\alpha\cdot e_1| + c^{-1}rC(|x|+r)^{\beta-2}  \Big|^{p-2} C(|x|+r)^{\beta-2}  \\
&\leq \frac{c^{p-2}}{|B_r|}\left(\int_{B_{2r}}\Big||\hat y\cdot e_1| + c^{-1}rC(|x|+r)^{\beta-2}  \Big|^{p-2} C(|x|+r)^{\beta-2} \dd y \right)+ o_{r+|x|}(1).
\end{split}
\end{equation}
 Once this is proved, it only remains to to prove that the first term in \eqref{eq:oufff} goes to zero. This is the estimate obtained on page 24 in the proof of Lemma A.4 in \cite{dTLi20}, with the small difference that we here integrate over $B_{2r}$ instead of $B_r$.

We now explain how to obtain \eqref{eq:oufff}. Fix $r$ and $x$ and consider the function
\[
f(y)=\Big||\hat y\cdot e_1| + c^{-1}rC(|x|+r)^{\beta-2}  \Big|^{p-2} C(|x|+r)^{\beta-2}. 
\]
The midpoint quadrature rule applied to $f(y)$ yields
\[
\begin{split}
\Big|  \int_{\tilde{B}_r}f(\hat{y}) dy -h^d  \sum_{y_\alpha\in B_r}f(\hat{y}_\alpha)\Big|
&\lesssim      \|D^2 f\|_{L^\infty(\tilde{B}_r)} h^2 |\tilde{B}_r|.
\end{split}
\]
Upon multiplying with $ \frac{c^{p-2}}{|B_r|}$, inserting $f(y)$ and rearranging, we obtain 
\[
\begin{split}
&\frac{c^{p-2} h^d}{|B_r|}  \sum_{y_\alpha\in B_r}\Big||\hat y_\alpha\cdot e_1| + c^{-1}rC(|x|+r)^{\beta-2}  \Big|^{p-2} C(|x|+r)^{\beta-2}\\
&\lesssim \frac{c^{p-2}}{|B_r|}\int_{\tilde B_r}\Big||\hat y\cdot e_1| + c^{-1}rC(|x|+r)^{\beta-2}  \Big|^{p-2} C(|x|+r)^{\beta-2} \dd y +   c^{p-2}\|D^2 f\|_{L^\infty(\tilde{B}_r)} h^2\\
&\lesssim \frac{c^{p-2}}{|B_r|}\int_{B_{2r}}\Big||\hat y\cdot e_1| + c^{-1}rC(|x|+r)^{\beta-2}  \Big|^{p-2} C(|x|+r)^{\beta-2} \dd y +   c^{p-2}\|D^2 f\|_{L^\infty(\tilde{B}_r)} h^2.
\end{split}
\]
where we use that $\tilde B_r \subset B_{r+\sqrt{d}h}\subset B_{2r}$ for $r$ small enough, since  $h=o(r)$. The only thing left is to prove that the last term is $o_{r+|x|}(1)$. Differentiation of $f$ yields
\[
\|D^2 f\|_{L^\infty(\tilde{B}_r)} \lesssim \Big|c^{-1}rC(|x|+r)^{\beta-2}  \Big|^{p-4 } C(|x|+r)^{\beta-2}.
\]
Since  $h=o(r^{p/(p-1)})=o(r^{3/2})$, $\beta-2>0$, $p-3<0$ and $c=\beta|x|^{\beta-1}$, we obtain
\[
\begin{split}
 c^{p-2} \|D^2 f\|_{L^\infty(\tilde{B}_r)} h^2 & \lesssim c^{p-2}c^{4-p}(|x|+r)^{(\beta-2)(p-3)} r^{p-1}\\
&= c^{2}(|x|+r)^{(\beta-2)(p-3)} r^{p-1}\\
 &=|x|^{2(\beta-1)}(|x|+r)^{(\beta-2)(p-3)} r^{p-1}\\
 \end{split}
\]
If $|x|\leq r$ then 
$$
|x|^{2(\beta-1)}(|x|+r)^{(\beta-2)(p-3)} r^{p-1}\lesssim r^{2(\beta-1)}r^{(\beta-2)(p-3)}r^{p-1} = r^{(\beta-1)(2+p-3)-p+3+p-1} = r^3.
$$
Likewise, if $r\leq |x|$ then
$$
|x|^{2(\beta-1)}(|x|+r)^{(\beta-2)(p-3)} r^{p-1}\lesssim |x|^{2(\beta-1)}|x|^{(\beta-2)(p-3)} |x|^{p-1}= |x|^3.
$$
This shows \eqref{eq:oufff} and concludes the proof.

\end{proof}

\bibliographystyle{abbrv}
%
%%If bibliography file is in the same folder use:
%
\bibliography{Bibliography} % The references (bibliography) information are stored in the file named "Bibliography.bib"

\begin{thebibliography}{10}

\bibitem{Amb03}
S.~Amghibech.
\newblock Eigenvalues of the discrete {$p$}-{L}aplacian for graphs.
\newblock {\em Ars Combin.}, 67:283--302, 2003.

\bibitem{AL16}
A.~Arroyo and J.~G. Llorente.
\newblock On the asymptotic mean value property for planar {$p$}-harmonic
  functions.
\newblock {\em Proc. Amer. Math. Soc.}, 144(9):3859--3868, 2016.

\bibitem{AR18}
A.~Attouchi and E.~Ruosteenoja.
\newblock Remarks on regularity for {$p$}-{L}aplacian type equations in
  non-divergence form.
\newblock {\em J. Differential Equations}, 265(5):1922--1961, 2018.

\bibitem{BS91}
G.~Barles and P.~E. Souganidis.
\newblock Convergence of approximation schemes for fully nonlinear second order
  equations.
\newblock {\em Asymptotic Anal.}, 4(3):271--283, 1991.

\bibitem{BL93}
J.~W. Barrett and W.~B. Liu.
\newblock Finite element approximation of the {$p$}-{L}aplacian.
\newblock {\em Math. Comp.}, 61(204):523--537, 1993.

\bibitem{BS18}
C.~Bucur and M.~Squassina.
\newblock An asymptotic expansion for the fractional $p$-laplacian and gradient
  dependent nonlocal operators.
\newblock {\em Commun. Contemp. Math. \emph{(online ready)}}, 2021.

\bibitem{CoLeMa17}
L.~Codenotti, M.~Lewicka, and J.~Manfredi.
\newblock Discrete approximations to the double-obstacle problem and optimal
  stopping of tug-of-war games.
\newblock {\em Trans. Amer. Math. Soc.}, 369(10):7387--7403, 2017.

\bibitem{dTEnJa18}
F.~del Teso, J.~Endal, and E.~R. Jakobsen.
\newblock Robust numerical methods for nonlocal (and local) equations of porous
  medium type. {P}art {II}: {S}chemes and experiments.
\newblock {\em SIAM J. Numer. Anal.}, 56(6):3611--3647, 2018.

\bibitem{dTEnJa19}
F.~del Teso, J.~Endal, and E.~R. Jakobsen.
\newblock Robust numerical methods for nonlocal (and local) equations of porous
  medium type. {P}art {I}: {T}heory.
\newblock {\em SIAM J. Numer. Anal.}, 57(5):2266--2299, 2019.

\bibitem{dTLi20}
F.~del Teso and E.~Lindgren.
\newblock A mean value formula for the variational $p$-laplacian.
\newblock {\em Preprint: arXiv:2003.07084}, 2020.

\bibitem{dTMP18}
F.~del Teso, J.~J. Manfredi, and M.~Parviainen.
\newblock Convergence of dynamic programming principles for the $p$-laplacian.
\newblock {\em Adv. Calc. Var. \emph{(online ready)}}, 2021.

\bibitem{Abd15}
A.~Elmoataz, M.~Toutain, and D.~Tenbrinck.
\newblock On the $p$-laplacian and $\infty$-laplacian on graphs with
  applications in image and data processing.
\newblock {\em SIAM Journal on Imaging Sciences}, 8(4):2412--2451, 2015.

\bibitem{FedPP-L}
R.~Ferreira, A.~de~Pablo, and M.~P\'{e}rez-Llanos.
\newblock Numerical blow-up for the {$p$}-{L}aplacian equation with a source.
\newblock {\em Comput. Methods Appl. Math.}, 5(2):137--154, 2005.

\bibitem{GM75}
R.~Glowinski and A.~Marrocco.
\newblock Sur l'approximation, par \'{e}l\'{e}ments finis d'ordre un, et la
  r\'{e}solution, par p\'{e}nalisation-dualit\'{e}, d'une classe de probl\`emes
  de {D}irichlet non lin\'{e}aires.
\newblock {\em Rev. Fran\c{c}aise Automat. Informat. Recherche
  Op\'{e}rationnelle S\'{e}r. Rouge Anal. Num\'{e}r.}, 9({\rm R}-2):41--76,
  1975.

\bibitem{equivalence2}
V.~Julin and P.~Juutinen.
\newblock A new proof for the equivalence of weak and viscosity solutions for
  the {$p$}-{L}aplace equation.
\newblock {\em Comm. Partial Differential Equations}, 37(5):934--946, 2012.

\bibitem{KMP12}
B.~Kawohl, J.~Manfredi, and M.~Parviainen.
\newblock Solutions of nonlinear {PDE}s in the sense of averages.
\newblock {\em J. Math. Pures Appl. (9)}, 97(2):173--188, 2012.

\bibitem{Kwa05}
K.~Y. Kim.
\newblock Error estimates for a mixed finite volume method for the
  {$p$}-{L}aplacian problem.
\newblock {\em Numer. Math.}, 101(1):121--142, 2005.

\bibitem{KKL}
J.~Korvenp\"{a}\"{a}, T.~Kuusi, and E.~Lindgren.
\newblock Equivalence of solutions to fractional {$p$}-{L}aplace type
  equations.
\newblock {\em J. Math. Pures Appl. (9)}, 132:1--26, 2019.

\bibitem{Le20}
M.~Lewicka.
\newblock Random tug of war games for the $p$-laplacian: $1<p<+\infty$.
\newblock {\em Preprint: arXiv:1810.03413v}, 2018.

\bibitem{LewMan14}
M.~Lewicka and J.~J. Manfredi.
\newblock Game theoretical methods in {PDE}s.
\newblock {\em Boll. Unione Mat. Ital.}, 7(3):211--216, 2014.

\bibitem{LeMa17}
M.~Lewicka and J.~J. Manfredi.
\newblock The obstacle problem for the {$p$}-laplacian via optimal stopping of
  tug-of-war games.
\newblock {\em Probab. Theory Related Fields}, 167(1-2):349--378, 2017.

\bibitem{Linq19}
P.~Lindqvist.
\newblock {\em Notes on the stationary {$p$}-{L}aplace equation}.
\newblock SpringerBriefs in Mathematics. Springer, Cham, 2019.

\bibitem{LM16}
P.~Lindqvist and J.~Manfredi.
\newblock On the mean value property for the {$p$}-{L}aplace equation in the
  plane.
\newblock {\em Proc. Amer. Math. Soc.}, 144(1):143--149, 2016.

\bibitem{WN01}
W.~Liu and N.~Yan.
\newblock Quasi-norm a priori and a posteriori error estimates for the
  nonconforming approximation of {$p$}-{L}aplacian.
\newblock {\em Numer. Math.}, 89(2):341--378, 2001.

\bibitem{WN02}
W.~Liu and N.~Yan.
\newblock On quasi-norm interpolation error estimation and a posteriori error
  estimates for {$p$}-{L}aplacian.
\newblock {\em SIAM J. Numer. Anal.}, 40(5):1870--1895, 2002.

\bibitem{LB93}
W.~B. Liu and J.~W. Barrett.
\newblock A remark on the regularity of the solutions of the {$p$}-{L}aplacian
  and its application to their finite element approximation.
\newblock {\em J. Math. Anal. Appl.}, 178(2):470--487, 1993.

\bibitem{Loi20}
S.~Loisel.
\newblock Efficient algorithms for solving the {$p$}-{L}aplacian in polynomial
  time.
\newblock {\em Numer. Math.}, 146(2):369--400, 2020.

\bibitem{MOS15}
J.~J. Manfredi, A.~M. Oberman, and A.~P. Sviridov.
\newblock Nonlinear elliptic partial differential equations and {$p$}-harmonic
  functions on graphs.
\newblock {\em Differential Integral Equations}, 28(1-2):79--102, 2015.

\bibitem{MRP}
J.~J. Manfredi, M.~Parviainen, and J.~D. Rossi.
\newblock An asymptotic mean value characterization for {$p$}-harmonic
  functions.
\newblock {\em Proc. Amer. Math. Soc.}, 138(3):881--889, 2010.

\bibitem{MPR12a}
J.~J. Manfredi, M.~Parviainen, and J.~D. Rossi.
\newblock Dynamic programming principle for tug-of-war games with noise.
\newblock {\em ESAIM Control Optim. Calc. Var.}, 18(1):81--90, 2012.

\bibitem{MPR12b}
J.~J. Manfredi, M.~Parviainen, and J.~D. Rossi.
\newblock On the definition and properties of {$p$}-harmonious functions.
\newblock {\em Ann. Sc. Norm. Super. Pisa Cl. Sci. (5)}, 11(2):215--241, 2012.

\bibitem{Obe05}
A.~M. Oberman.
\newblock A convergent difference scheme for the infinity {L}aplacian:
  construction of absolutely minimizing {L}ipschitz extensions.
\newblock {\em Math. Comp.}, 74(251):1217--1230, 2005.

\bibitem{Ob2}
A.~M. Oberman.
\newblock Convergent difference schemes for degenerate elliptic and parabolic
  equations: {H}amilton-{J}acobi equations and free boundary problems.
\newblock {\em SIAM J. Numer. Anal.}, 44(2):879--895, 2006.

\bibitem{ObermanpLap}
A.~M. Oberman.
\newblock Finite difference methods for the infinity {L}aplace and
  {$p$}-{L}aplace equations.
\newblock {\em J. Comput. Appl. Math.}, 254:65--80, 2013.

\bibitem{Pe09}
Y.~Peres, O.~Schramm, S.~Sheffield, and D.~B. Wilson.
\newblock Tug-of-war and the infinity {L}aplacian.
\newblock {\em J. Amer. Math. Soc.}, 22(1):167--210, 2009.

\bibitem{Wang21}
Y.-Z. Wang and H.~Huang.
\newblock Eigenvalue estimates of the p-laplacian on finite graphs.
\newblock {\em Differential Geometry and its Applications}, 74:101697, 2021.

\bibitem{Yam79}
M.~Yamasaki.
\newblock Discrete potentials on an infinite network.
\newblock {\em Mem. Fac. Lit. Sci. Shimane Univ.}, 13:31--44, 1979.

\end{thebibliography}

\end{document}